\documentclass[11pt]{amsart}
\usepackage{amsfonts, amsbsy, amsmath, amsthm, amssymb, latexsym, verbatim, enumerate}
\usepackage{mathrsfs}
\usepackage[top=30mm,right=30mm,bottom=30mm,left=30mm]{geometry}

\usepackage{bm}
\usepackage[all]{xy}

\usepackage[pdftex]{color,graphicx}

\makeatletter
\newcommand{\imod}[1]{\allowbreak\mkern4mu({\operator@font mod}\,\,#1)}
\makeatother

\newtheorem{propo}{Proposition}[section]
\newtheorem{proposition}[propo]{Proposition}
\newtheorem{defi}[propo]{Definition}

\newtheorem{lemma}[propo]{Lemma}
\newtheorem{corol}[propo]{Corollary}

\newtheorem{theor}[propo]{Theorem}
\newtheorem{theorem}[propo]{Theorem}

\numberwithin{equation}{section}

\newcommand{\Ker}{\operatorname{Ker}}
\newcommand{\Aut}{{\mathrm {Aut}}}

\newcommand{\Irr}{{\mathrm {Irr}}}

\newcommand{\Ind}{{\mathrm {Ind}}}
\newcommand{\diag}{{\mathrm {diag}}}

\newcommand{\Hom}{{\mathrm {Hom}}}

\def\rank{\mathop{\mathrm{ rank}}\nolimits}

\newcommand{\Tr}{{\mathrm {Tr}}}
\newcommand{\tr}{{\mathrm {tr}}}

\newcommand{\Spec}{{\mathrm {Spec}}\,}

\newcommand{\CC}{{\mathbb C}}

\newcommand{\ZZ}{{\mathbb Z}}

\newcommand{\FF}{{\mathbb F}}
\newcommand{\FQ}{\mathbb{F}_{q}}

\newcommand{\AC}{\mathcal{A}}

\newcommand{\EC}{\mathcal{E}}

\newcommand{\HC}{{\underline{\boldsymbol{H}}}}

\newcommand{\SC}{\mathcal{S}}

\newcommand{\CL}{\mathcal{C}}
\newcommand{\CLS}{\mathcal{C}^*}
\newcommand{\CLF}{\mathfrak{C}}
\newcommand{\OC}{\mathcal{O}}
\newcommand{\XC}{\mathcal{X}}

\newcommand{\ZB}{\mathbf{Z}}
\newcommand{\CB}{\mathbf{C}}

\newcommand{\OB}{\mathbf{O}}
\newcommand{\varep}{\varepsilon}
\newcommand{\eps}{\epsilon}
\newcommand{\lam}{\lambda}

\newcommand{\al}{\alpha}
\newcommand{\gam}{\gamma}
\newcommand{\Csum}{\boldsymbol{\sigma}}
\newcommand{\Clin}{\boldsymbol{\lambda}}
\newcommand{\om}{\omega}
\newcommand{\oms}{\omega^*}

\newcommand{\Om}{\Omega}

\newcommand{\GL}{\mathrm {GL}}
\newcommand{\SL}{\mathrm {SL}}

\newcommand{\GU}{\mathrm {GU}}
\newcommand{\SU}{\mathrm {SU}}

\newcommand{\Sp}{\mathrm {Sp}}
\newcommand{\CSp}{\mathrm {CSp}}
\newcommand{\PCSp}{\mathrm {PCSp}}
\newcommand{\SO}{\mathrm {SO}}
\newcommand{\GO}{\mathrm {GO}}

\newcommand{\DC}{D^{\circ}}
\newcommand{\Stab}{{\mathrm {Stab}}}

\newcommand{\dl}{{\mathfrak d}}

\newcommand{\blc}{{\mathfrak b}_{\sf C}}
\newcommand{\bld}{{\mathfrak b}_{\sf {BD}}}

\newcommand{\Spin}{{\mathrm{Spin}}}

\newcommand{\reg}{{\mathbf{reg}}}
\newcommand{\cl}{\mathfrak{l}}
\newcommand{\bolds}{{\boldsymbol{s}}}

\newcommand{\tw}[1]{{}^#1\!}

\newcommand{\SR}{\tw* R}

\renewcommand{\mod}{\bmod \,}
\newcommand{\cC}{\mathcal{C}}
\newcommand{\cG}{\mathcal{G}}
\newcommand{\cS}{\mathcal{S}}
\newcommand{\cU}{\mathcal{U}}
\newcommand{\cW}{\mathcal{W}}
\newcommand{\cX}{\mathcal{X}}

\newcommand{\calx}{\bar x}
\newcommand{\uC}{{\underline{\boldsymbol{C}}}}
\newcommand{\uG}{{\underline{\boldsymbol{G}}}}
\newcommand{\uX}{{\underline{\boldsymbol{X}}}}
\newcommand{\uB}{{\underline{\boldsymbol{B}}}} 
\newcommand{\QF}{\mathsf {Q}}
\newcommand{\sfr}{\mathsf {r}}

\begin{document}

\title{Character Levels and Character Bounds. II}

\author{Robert M. Guralnick}
\address{Department of Mathematics, University of Southern California,
Los Angeles, CA 90089-2532, USA}
\email{guralnic@math.usc.edu}

\author{Michael Larsen}
\address{Department of Mathematics\\
    Indiana University \\
    Bloomington, IN 47405\\
    U.S.A.}
\email{larsen@math.indiana.edu}

\author{Pham Huu Tiep}
\address{Department of Mathematics\\ 
Rutgers University\\ Piscataway, NJ 08854\\    
U. S. A.} 
\email{tiep@math.rutgers.edu}


\thanks{The first author was partially supported by the NSF
grant DMS-1600056. The second author was partially supported by the NSF 
grant DMS-1702152. The third author was partially supported by the NSF grant DMS-1840702.}
\thanks{The paper is partially based upon work 
supported by the NSF under grant DMS-1440140 while the first and the third authors were in residence at 
the Mathematical Sciences Research Institute in Berkeley, California, during the Spring 2018
semester. It is a pleasure to thank the Institute for support, hospitality, and stimulating environments.}
\thanks{The authors are grateful to Martin Liebeck and Geoffrey Robinson for helpful conversations.}

\begin{abstract} 
This paper is a continuation of \cite{GLT}, which develops a level theory and establishes 
strong character bounds for finite simple groups of linear and unitary type in the case that the centralizer
of the element has small order compared to $|G|$ in a logarithmic
sense.  We strengthen the results of \cite{GLT} and extend them to all groups of classical type. 
\end{abstract}

\maketitle

\tableofcontents

\section{Introduction}
Let $G$ be a finite group and $\chi$ an irreducible character.  For all $g\in G$ we have the trivial bound $|\chi(g)| \le \chi(1)$,
but stronger bounds typically hold, at least for most elements and for most characters.  The  centralizer bound
$|\CB_G(g)|^{1/2}$, which follows immediately from Schur's lemma, is often much better than $\chi(1)$.  In particular,
good bounds are most
easily obtained for elements with $|\CB_G(g)| \ll |G|$ and characters with $\chi(1)$ not too much smaller than $|G|^{1/2}$.
For symmetric groups, non-trivial bounds can be found in \cite{LaS,RS,Ro}; see also the references within.

For groups of Lie type, the best general result, due to Gluck \cite{Gl}, is rather weak. It was improved for all elements whose
support is bounded below in \cite[Theorem 1.2.1]{LST}.  At the opposite extreme, it is well known (see, e.g., \cite[(4.26.1)]{L}) that character values on regular semisimple elements can
be bounded above by a bound depending only on the rank of $G$.
Recently Bezrukavnikov, Liebeck, Shalev, and Tiep \cite{BLST}
gave a  sharp bound for general elements (that satisfy a mild condition on their centralizers).  

This paper is a continuation of \cite{GLT}, which gives
strong estimates for finite simple groups of linear and unitary type in the case that $|\CB_G(g)|$ is small compared to $|G|$ in a logarithmic
sense.  We strengthen the results of \cite{GLT} and extend them to all groups of classical type.  Together these two papers can be regarded
as giving the counterpart for groups of Lie type of the main result of \cite{LaS} for symmetric groups.

The idea behind the proof is that, starting with a larger than expected value $|\chi(g)|$, we replace $\chi$ by $\chi^m$ for a suitable positive integer $m$ 
to amplify the effect.
If $\chi^m$ decomposes into a manageable number of irreducible factors, one can show that for at least one of these factors, $g$ violates the centralizer bound.
One needs to show that if $\chi(1)$ is small compared to $|G|^{1/2}$ (and it is only when $\chi(1)$ is small that the centralizer bound is inadequate), for
suitable values of $m$, the total number of irreducible factors of $\chi^m$ is small. 

In \cite{GLT}, we develop a theory of \emph{levels} for irreducible representations for $\SL_n(q)$ and $\SU_n(q)$.  Low dimensional representations have low level, and the tensor product of two representations of low level decomposes into a controllable number of irreducible factors.  We then restrict any low degree character $\chi$ of a symplectic or  orthogonal group to a Levi subgroup of type $\GL_n(q)$ and show that the resulting character has a controllable number of irreducible factors.  We can then use results for $\GL_n(q)$ to bound the number of irreducible factors in $\chi^m$, and subsequently derive the desired 
character bounds for all finite classical groups. 

In parallel to \cite{GLT}, we also develop a level theory for representations of orthogonal and symplectic groups.
Note that in \cite{GH1}, \cite{GH2} Gurevich and Howe describe a {\it $U$-rank} theory for classical groups which may have some parallels with 
our level theory. As shown in \cite[Theorem 9.8]{GLT}, the $U$-rank is related to, but coarser than the level of irreducible characters. In this paper, 
we prove several results relating the level, the degree, and the $U$-rank, for classical groups not of type $A$. 

To formulate our main results more precisely, it is convenient to start with some definitions. 

\begin{defi}
\label{classical}
{\em Fix a constant $a > 0$, a prime power $q$, and an integer $n \geq 2$. 

\begin{enumerate}[\rm(i)]
\item A {\it $(q,n,a)$-classical group} is a finite group $G$ such that
the last term $G^{(\infty)}$ of its derived series satisfies the following two conditions.

\begin{enumerate}[\rm (a)]
\item $G^{(\infty)}$ is of the form $\tilde G/Z$ where $Z \leq \ZB(\tilde G)$, and $\tilde G$ is one of the following groups:
$\SL_n(q)$, $\SU_n(q)$, $\Sp_{2n}(q)$, $\Omega_{2n+1}(q)$ with $2 \nmid q$, $\Omega^+_{2n}(q)$, or $\Omega^-_{2n+2}(q)$.
\item $[G:G^{(\infty)}] \leq q^a$.
\end{enumerate}

\item $\CL_n(q)$ denotes the following collection of finite groups:
$\GL_n(q)$, $\SL_n(q)$, $\GU_n(q)$, $\SU_n(q)$, $\Sp_{2n}(q)$, $\SO^{\pm}_{2n}(q)$, (the full orthogonal groups) 
$\GO^{\pm}_{2n}(q)$, and $\SO_{2n+1}(q)$ for odd $q$.

\item $\CLS_n(q)$ denotes the following collection of finite groups:
$\GL_n(q)$, $\SL_n(q)$, $\GU_n(q)$, $\SU_n(q)$, $\Sp_{2n}(q)$, $\SO^+_{2n}(q)$, $\Spin^+_{2n}(q)$,
$\SO^-_{2n+2}(q)$, $\Spin^-_{2n+2}(q)$, and $\SO_{2n+1}(q)$ and $\Spin_{2n+1}(q)$ for odd $q$, and their quotients
by central subgroups.
\end{enumerate}
}
\end{defi}

\begin{defi}
{\em 

\begin{enumerate}[\rm(i)]
\item A {\it classical group} (with parameter $n \geq 2$) is a $(q,n,a)$-classical group for some prime power $q$ and a fixed 
$a > 0$, which will be taken to be $4$ henceforth. 

\item A \emph{spin group} $G$ is a finite group of the form $\Spin_{2n+1}(q)$, $n\ge 3$, or $\Spin_{2n}^{\pm}(q)$, $n\ge 4$.
\end{enumerate}
}
\end{defi}

In particular, all finite simple groups $S$ of Lie type $A_{n-1}$, $^2A_{n-1}$, $B_n$, $C_n$, $D_n$, and $^2D_n$, as well as 
{\it almost simple} groups $G$ with $S \lhd G \leq \Aut(S)$, are 
classical groups in our sense.

The main theorem of this paper is the following:

\begin{theor}\label{main1}
For every $\varep>0$, there exists $\delta > 0$ such that the following statement holds.
If $G$ is a classical group or a spin group and $g\in G$ satisfies
$|\CB_G(g)| \leq |G|^\delta$, then
$$|\chi(g)| \leq \chi(1)^\varep$$
for all $\chi \in \Irr(G)$.
\end{theor}

The proof given could in principle be made effective but with very bad constants.
For $\varep > 4/5$, we have an effective version of this theorem with a somewhat different proof, giving substantially better bounds,
as follows:

\begin{theor}\label{main2}
For every $\varep$ with $4/5 < \varep < 1$, there exists an explicit constant $\delta > 0$ such that the following statement holds.
Let $q$ be any prime power, $n \geq 9$, and let $G \in \CLS_n(q)$.
Suppose that $g \in G$ satisfies $|\CB_G(g)| \leq q^{n^2\delta}$. Then
$$|\chi(g)| \leq \chi(1)^\varep$$
for all $\chi \in \Irr(G)$. 
\end{theor}

For instance, if $\varep = 0.992$, 
then one can take $\delta = 0.0011$.

In fact, we deduce Theorem \ref{main2} from Theorem \ref{main2a}, which yields a slightly better constant $\delta$, but at the price of 
having an extra factor $4$ in the character bound.

Our next main result bounds the degree of any irreducible character of any given level (where the level $\cl(\chi)$ is 
defined below in Definition \ref{def-level}).

\begin{theor}\label{main3}
Let $q$ be a prime power and let $G$ be one of the following classical groups: $\Sp_{2n}(q)$ with $2 \nmid q$ and $n \geq 1$,
$\Sp_{2n}(q)$ with $2|q$ and $n \geq 2$, or $\Omega^\pm_n(q)$ with $n \geq 6$. Let $\chi \in \Irr(G)$ be of level $\ell = \cl(\chi)$. Then 
the following statements hold for $k := \lfloor (\ell+2)/3 \rfloor$.

\begin{enumerate}[\rm(i)]
\item If $G = \Sp_{2n}(q)$ with $2 \nmid q$, then 
$$q^{nk-k(k+1)/2}\biggl(\frac{q-1}{2}\biggr)^k \leq \chi(1) \leq \biggl(\frac{q^n+1}{2}\biggr)^\ell.$$
\item If $G = \Sp_{2n}(q)$ with $2|q$, then 
$$q^{2nk-k(2k+1)}\biggl(\frac{(q-1)^2}{2}\biggr)^k \leq \chi(1) \leq \biggl(\frac{q^{2n}-1}{q-1}\biggr)^\ell.$$
\item If $G = \Omega^\pm_n(q)$ and $(n,q) \neq (8,2)$, $(9,2)$, then 
$$q^{nk-2k(k+1)}(q-1)^k \leq \chi(1) \leq \biggl(\frac{q^n-1}{q-1}\biggr)^\ell.$$
\end{enumerate}
\end{theor}

Theorems \ref{main1} and \ref{main2} are expected to be useful in a number of applications. They are already used in \cite{LST}. 
Here we offer two applications, one on random walks, see Corollary \ref{walk}, and another on product one subvarieties in simple 
algebraic groups, see Theorem \ref{product-one}.

\section{Preliminaries}

The following observation goes back (at least) to Burnside:

\begin{lemma}\label{value} {\rm \cite[Lemma 2.2]{GLT}}
Let $\Theta$ be a generalized character of a finite group $G$ which takes exactly $N$ different values $a_0 = \Theta(1)$, 
$a_1, \ldots ,a_{N-1}$ on $G$. Suppose also that $\Theta(g) \neq \Theta(1)$ for all $1 \neq g \in G$. Then every irreducible
character $\chi$ of $G$ occurs as an irreducible constituent of $\Theta^k$ for some $0 \leq k \leq N-1$.
\end{lemma}

\begin{lemma}\label{coset}
Let $H$ be a subgroup of a finite group $G$ and let $\lam \in \Irr(H)$ be of degree $1$. Then 
$$[\Ind^G_H(\lam),\Ind^G_H(\lam)]_G \leq [\Ind^G_H(1_H),\Ind^G_H(1_H)]_G = |H\backslash G/H|.$$
\end{lemma}

\begin{proof}
An application of Mackey's formula.
\end{proof}

For any (not necessarily irreducible) character $\rho$ of a finite group $G$, let 
\begin{equation}\label{d-sum1}
  \Csum(\rho,G) := \sum_{\chi \in \Irr(G)}[\rho,\chi]_G
\end{equation}
denote the sum of all multiplicities of irreducible constituents of $\rho$, where $[\cdot,\cdot]_G$ is the 
usual scalar product of (complex-valued) class functions on $G$. Also, let
\begin{equation}\label{d-sum2}
  \Clin(\rho,G) := \sum_{\chi \in \Irr(G),~\chi(1) = 1}[\rho,\chi]_G
\end{equation}  
denote the sum of all multiplicities of {\it linear} irreducible constituents of $\rho$. Then we have the following 
elementary properties:

\begin{lemma}\label{mults}
Let $H$ be a subgroup of a finite group $G$, and let $\rho$ be a complex character of $G$ and $\varphi$ be a complex 
character of $H$. Then the following inequalities hold.
\begin{enumerate}[\rm(i)]
\item
$$\begin{aligned}
\Clin(\rho,G) & \leq \Csum(\rho,G) \leq [\rho,\rho]_G,\\
\Csum(\rho,H) & \leq \Csum(\rho,G) \cdot \max_{\al \in \Irr(G),~[\rho,\al]_G > 0}\Csum(\al,H),\\
\Clin(\rho,H) & \leq \Csum(\rho,G) \cdot \max_{\al \in \Irr(G),~[\rho,\al]_G > 0}\Clin(\al,H).
\end{aligned}$$
\item
$$\begin{aligned}
\Csum(\rho,G) & \leq \Csum(\rho,H) \leq \Csum(\rho,G)[G:H],\\
\Csum(\varphi,H) & \leq \Csum(\Ind^G_H(\varphi),G) \leq \Csum(\varphi,H)[G:H].
\end{aligned}$$
\end{enumerate}
\end{lemma}

\begin{proof}
We will prove only the non-obvious ones among these inequalities. For (ii), without any loss we may assume that
$\rho \in \Irr(G)$ and decompose $\rho|_H = \sum_i n_i\al_i$, where $\al_i \in \Irr(H)$ are pairwise distinct. Then
$$n_i = [\rho|_H,\al_i]_H = [\rho,\Ind^G_H(\al_i)]_G \leq \frac{\al_i(1)[G:H]}{\rho(1)}$$
by degree consideration. Hence
$$\Csum(\rho,H) \leq \sum_i n_i^2 \leq  \frac{\sum_in_i\al_i(1)[G:H]}{\rho(1)} = [G:H].$$
Next, without loss we may also assume that
$\varphi \in \Irr(H)$ and decompose $\Ind^G_H(\varphi) = \sum_i m_i\beta_i$, where $\beta_i \in \Irr(G)$ are pairwise distinct. Then
$$m_i = [\Ind^G_H(\varphi),\beta_i]_G = [\varphi,(\beta_i)|_H]_H \leq \frac{\beta_i(1)}{\varphi(1)}$$
again by degree consideration. Hence
$$\Csum(\Ind^G_H(\varphi),G) \leq \sum_i m_i^2 \leq  \frac{\sum_im_i\beta_i(1)}{\varphi(1)} = [G:H].$$
\end{proof}

\begin{lemma}\label{abelian}
Let $G = A \rtimes H$ be a split extension of a normal abelian subgroup $A$ and a subgroup $H$. Then for any 
$\chi \in \Irr(G)$ and any $\lam \in \Irr(H)$ with $\lam(1) = 1$, we have 
$$[\chi|_H,\lam]_H \leq 1.$$
\end{lemma}

\begin{proof}
The statement is obvious if $A \leq \Ker(\chi)$. Assume otherwise and consider a nontrivial irreducible constituent 
$\al$ of $\chi|_A$. If $T := \Stab_G(\al) = A \rtimes J$, then by Clifford's theorem
$$\chi = \Ind^G_T(\hat\al)$$
for some $\hat\al \in \Irr(T)$ lying above $\al$. 
Note that $G=TH$, so applying Mackey's formula we get
$$\chi|_H = (\Ind^G_T(\hat\al))|_H = \Ind^H_{T \cap H}(\hat\al|_{T \cap H}) = \Ind^H_J(\hat\al|_J).$$
Next, observe that $\Ker(\al) \lhd T$, and 
$$T/\Ker(\al) = (A/\Ker(\al)) \times J,$$
whence $\hat\al|_J$ is irreducible. It follows that
$$[\chi|_H,\lam]_H = [\Ind^H_J(\hat\al|_J),\lam]_H = [\hat\al|_J,\lam|_J]_J \leq 1.$$
\end{proof}

\begin{lemma}\label{orbits1} {\rm  \cite[Lemma 2.4]{GLT}}
Let $(G,Q)$ be either $(\GL_n(q),q)$ or $(\GU_n(q),q^2)$, and let $V = \FF_Q^n$ denote the natural module for $G$.
Then, for any $1 \leq j \leq n$, the number $N_j$ of $G$-orbits on the set $\Omega_j$ of ordered  $j$-tuples $(v_1, \ldots, v_j)$
with $v_i \in V$ is at most $8q^{j^2/4}$ in the first case, and at most $2q^{j^2}$ in the second case.  
\end{lemma}

\begin{lemma}\label{orbits2}
Let $V = \FF_q^n$ be endowed with either a non-degenerate symplectic form 
$(\cdot,\cdot)$, or a quadratic form $Q$ associated with a non-degenerate 
symmetric bilinear form $(\cdot,\cdot)$ (which is assumed to be alternating if $2|q$). Set $\eps = -1$ in the former case,
and $\eps = +1$ in the latter case.
Let $G = \Sp(V)$ be the full group of isometries of $(\cdot,\cdot)$ on $V$
in the former case, and $G = \GO(V)$ be the full group of isometries of $Q$ in the latter case.
Then, for any $1 \leq j \leq n$, the number $N_j$ of $G$-orbits on the set $\Omega_j$ of ordered  $j$-tuples $(v_1, \ldots, v_j)$
with $v_i \in V$ is less than $6q^{j(j+\eps)/2}$.  
\end{lemma}

\begin{proof}
(i) Consider $U = \FF_q^j$ with a fixed basis $(e_1, \ldots ,e_j)$. Then there is a natural bijection between $\Omega_j$
and $\Hom(U,V)$: any $\varpi = (v_1, \ldots,v_j)$ corresponds to $f=f_\varpi \in \Hom(U,V)$ with $f(e_i) = v_i$. Suppose that 
$\varpi' = g(\varpi)$ for some $g \in G$. Then $f_{\varpi'} = gf_\varpi$ and $\Ker(f_{\varpi'}) = \Ker(f_\varpi)$. Furthermore, the bilinear form $(\cdot,\cdot)$
of $V$ restricted to $f_\varpi(V)$ and $f_{\varpi'}(V)$ have the same Gram matrices in
the bases $(f_\varpi(u_1), \ldots,f_\varpi(u_k))$ and $(f_{\varpi'}(u_1), \ldots ,f_{\varpi'}(u_k))$, if $(u_1, \ldots,u_k)$ is a basis 
of  $U/\Ker(f_\varpi)$. Moreover, in the case $G = \GO(Q,V)$, we also have 
$Q(f_\varpi(u_i)) = Q(f_{\varpi'}(u_i))$ for $1 \leq i \leq k$. We will refer to this as $f_\varpi$ and $f_{\varpi'}$ having the same {\it isometric data}.
For the Gram matrices in the basis $(f_\varpi(u_1), \ldots,f_\varpi(u_k))$, there are at most 
$q^{k(k-1)/2}$ possibilities in the symplectic case, as well as in the quadratic case but with $2|q$, and 
at most $q^{k(k+1)/2}$ possibilities in the quadratic case with $2 \nmid q$. Adding $q^k$ possibilities for 
$Q(f_\varpi(u_i))$, $1 \leq i \leq k$, in the quadratic case with $2|q$, we see that there are at most 
$q^{k(k+\eps)/2}$ possible isometric data when $k$ is fixed.

Conversely, assume that $f_\varpi$ and $f_{\varpi'}$ have the same kernel $W$ for some $\varpi,\varpi' \in \Omega_j$.
Again we fix a basis $(u_1, \ldots,u_k)$ of $U/W$, and assume in addition that $f_\varpi$ and $f_{\varpi'}$ have 
the same isometric data. By Witt's lemma \cite[p. 81]{A}, 
there is some $g \in G$ such that $g(f_\varpi(u_i)) = f_{\varpi'}(u_i)$ for all $1 \leq i \leq k$. Hence $f_{\varpi'} =gf_\varpi$ and 
so $\varpi' = g(\varpi)$.

\smallskip
(ii) We have shown that $N_j$ is at most the sum over $k$ of 
the total number of $j-k$-dimensional subspaces $W$ in $U$ weighted by 
a factor of $q^{k(k+\eps)/2}$, i.e.
$$N_j \leq \sum_{i=0}^j  q^{i(i+1)/2}\binom ji_q,$$
where $\binom ji_q$ denotes the Gaussian binomial coefficient:
$$\binom ji_q := \frac{\prod^{i-1}_{t=0}(q^j-q^t)}{\prod^{i-1}_{t=0}(q^i-q^t)}.$$ 
By \cite[Lemma 4.1(i)]{LMT} we have
$$\binom ji_q  = \frac{\prod^{i-1}_{t=0}(q^j-q^t)}{\prod^{i-1}_{t=0}(q^i-q^t)} = 
   q^{i(j-i)}\frac{\prod^{i-1}_{t=0}(1-1/q^{j-t})}{\prod^{i-1}_{t=0}(1-1/q^{i-t})} < 
    \frac{32}{9}q^{i(j-i)}.$$
It follows that 
$$N_j  < \frac{32}{9}\sum^j_{i=0}q^{i(j-(i-\eps)/2)} < 
\frac{32}{9}q^{j(j+\eps)/2}\sum^{\infty}_{i=0}\frac{1}{q^{i(i+1)/2}} < 6q^{j(j+\eps)/2},
$$   
since
$$\sum^{\infty}_{i=0}\frac{1}{q^{i(i+1)/2}} <  1 + \frac{1}{q} + \frac{1}{q^3} + \sum^{\infty}_{t=6}\frac{1}{q^t} \leq \frac{53}{32}.$$
The argument above also shows that $N_j \geq q^{j(j+\eps)/2}$ if $j \leq n/2$, as $G$ has at least $q^{j(j+\eps)/2}$ possible isometric data, hence orbits, on linearly independent $j$-tuples.
\end{proof}

\section{Character level for finite classical groups}\label{level-classical}
Let $A = \FF_q^n$ be endowed with a non-degenerate, symmetric or alternating, bilinear form $(\cdot,\cdot)$, and 
possibly also with a quadratic form $\QF$ associated with $(\cdot,\cdot)$ if $2|q$. The group of isometries of either 
the form $(\cdot,\cdot)$ or of the quadratic form $\QF$ is a classical group on $A$, and its action on the point set of
$A$ affords the permutation character
\begin{equation}\label{tau1}
  \tau_A: g \mapsto |\CB_A(g)| = q^{\dim_{\FF_q}\Ker(g-1_A)}.
\end{equation}  
Correspondingly, we can consider $\Sp(A)$ or $\GO(A)$, and throughout this section $\Sp(A)$ and $\GO(A)$ will denote
such a group. One can also embed such a classical group $G$ in the unitary group $\GU_n(q)$ defined by a suitable Hermitian form on 
$A \otimes_{\FQ} \FF_{q^2}$,
and the restriction to $G$ of the {\it reducible Weil character $\zeta_n$} of $\GU_n(q)$, cf. \cite[\S4]{TZ2}, yields a (reducible) character
\begin{equation}\label{zeta1}
  \zeta_A:  g \mapsto (-1)^n(-q)^{\dim_{\FF_q}\Ker(g-1_A)}.
\end{equation}

As in \cite{LBST1} and \cite{GLT}, we will explore certain {\it dual pairs} $G \times S$,
where $G$ and $S$ are certain finite classical groups. Given any character $\om$ of a group $\Gamma$ with 
a fixed homomorphism $G \times S \to \Gamma$, we can decompose
$$\om|_{G \times S} = \sum_{\al \in \Irr(S)} D_\al \otimes \al,$$
where $D_\al$ is either zero or a $G$-character, and its value at any $g \in G$ is given by the formula
\begin{equation}\label{dual1}
  D_\al(g) = \frac{1}{|S|}\sum_{s \in S}\om(gs)\bar\al(s),
\end{equation}  
see \cite[Lemma 5.5]{LBST1}. 

The quadruples $(G,S,\Gamma,\om)$ we consider in this paper are as follows.
We assume $2 \nmid q$, and endow $A := \FF_q^{2n}$ with a non-degenerate symplectic form 
$(\cdot,\cdot)$. Also, consider $B := \FF_q^m$ with a non-degenerate symmetric 
bilinear form $(\cdot,\cdot)$. Then the formula
$$(a \otimes b,a' \otimes b') = (a,a')(b,b')$$
for $a \in A$ and $b \in B$, extended by bilinearity, 
defines a non-degenerate symplectic form on $W := A \otimes_{\FQ} B \cong \FF_q^{2nm}$. This yields 
a homomorphism $G \times S \to \Gamma := \Sp(W)$, where either 
\begin{equation}\label{pairs}
\begin{array}{lll}
{\rm (a)} & G = \Sp(A) \cong \Sp_{2n}(q), & S = \GO(B) \cong \GO^\pm_m(q),\mbox{ or}\\
{\rm (b)} & G = \SO(B) \cong \SO^\pm_m(q), & S = \Sp(A) \cong \Sp_{2n}(q).
\end{array}
\end{equation}
Next, $\om$ is one of the two {\it reducible Weil characters} $\om_{nm}$ and $\oms_{nm}$ of
degree $q^{nm}$ of $\Gamma \cong \Sp_{2nm}(q)$, see e.g. \cite[\S1]{GMT} for their definition.

We recall the following properties of these Weil characters $\om_n$ and $\oms_n$ of $\Sp_{2n}(q)$:

\begin{propo}\label{weil-sp}
Let $q$ be an odd prime power and let $\om_n$ and $\oms_n$ be the two reducible Weil characters of 
$G := \Sp_{2n}(q)$, with natural module $A := \FF_q^{2n}$. Then the following statements hold.
\begin{enumerate}[\rm(i)]
\item For any $g \in G$, $|\om_n(g)|^2 = |\oms_n(g)|^2 = \tau_A(g)= |\CB_A(g)| = q^{\dim_{\FQ}\Ker(g-1_A)}$.
\item If $q \equiv 1 (\mod 4)$ then 
$$\om_n = \overline{\om_n},~\oms_n = \overline{\oms_n},~(\om_n)^2 = (\oms_n)^2 = \tau_A,~\om_n\oms_n = \zeta_A,$$
and if $q \equiv 3 (\mod 4)$ then
$$\oms_n = \overline{\om_n},(\om_n)^2 = (\oms_n)^2 = \zeta_A,~\om_n\om_n^* = \tau_A,$$
where $\tau_A$ and $\zeta_A$ are as defined in \eqref{tau1} and \eqref{zeta1}.
\item In the situation of \eqref{pairs}{\rm (a)}, the restriction of $\om$ to $G$ is $\om_n^m$ or $\om_n^{m-1}\oms_n$.
Furthermore, the restriction of $\om$ to $\SO(B)$ is $(\tau_B)^n$, where $\tau_B$ denotes the permutation character of $\SO(B)$ 
acting on the point set of $B$.
\end{enumerate} 
\end{propo}

\begin{proof}
For (i), see Theorem 2.1 and Lemma 2.2 of \cite{GMT}. Parts (ii) and (iii) were proved in \cite{T1}, see also \cite[\S5]{MT}. 
\end{proof}

\begin{defi}\label{def-level}
{\em 
\begin{enumerate}[\rm(i)]
\item Let $q$ be an {\em odd} prime power and $G = \Sp_{2n}(q)$. The {\it level} $\cl(\chi)$ of an irreducible 
character $\chi \in \Irr(G)$ is defined to be the smallest non-negative integer $k$ such that $\chi$ is an irreducible 
constituent of $(\om_n+\oms_n)^k$.
\item Let $\Om(A) \leq G \leq \GO(A)$ with $A = \FF_q^n$, or $G = \Sp(A)$ with $A = \FF_q^{2n}$ and $2|q$. The {\it level} $\cl(\chi)$ of an irreducible character $\chi \in \Irr(G)$ is defined to be the smallest non-negative integer $k$ such that $\chi$ is an irreducible 
constituent of $(\tau_A+\zeta_A)^k$.
\end{enumerate}
}
\end{defi}

\begin{lemma}\label{even}
Let $G = \SO(A) = \SO^\pm_n(q)$ if $2 \nmid q$ and $G = \Omega(A) = \Omega^\pm_{2m}(q)$ if $2|q$. 
Then $\zeta_A = \tau_A$ on $G$. 
\end{lemma}

\begin{proof}
Let $d(g) := \dim\Ker(g-1_A)$ for any $g \in G$. Also, let $s$ denote the semisimple part of $g$. First we consider the 
case $2 \nmid q$. Then we can decompose the $s$-module $A$ as orthogonal sum $A_+ \oplus A_- \oplus A_0$, 
where $A_+ = \Ker(s-1_A)$, $A_- = \Ker(s+1_A)$, and $s$ has no eigenvalue $1$ or $-1$ on $A_0$. 
Since $g \in \SO(A)$, $\dim A_-$ and $\dim A_0$ are both even, whence $d \equiv n (\mod 2)$ for  
$d:=\dim A_+$. Using the description of unipotent elements in \cite[\S13.4]{C}, one can show that 
the total number $e$ of Jordan blocks of of the unipotent element $g|_{A_+}$ is congruent to $d$ modulo $2$.  
Since $d(g) = e$, we conclude that $d(g) \equiv n (\mod 2)$ and so $\tau_A(g) = \zeta_A(g)$ by \eqref{tau1} and 
\eqref{zeta1}, as stated. In the case $2|q$ we have $2|d(g)$ for all $g \in \Omega(A)$ by \cite[Lemma 5.8(ii)]{GT2},
and so we are done. 
\end{proof}

\begin{lemma}\label{level-range}
The following statements hold.
\begin{enumerate}[\rm(i)]
\item Let $q$ be an odd prime power and $G = \Sp_{2n}(q)$. Then $0 \leq \cl(\chi) \leq 2n+1$ for any $\chi \in \Irr(G)$.
\item Let $2|q$ and $G = \Sp_{2m}(q)$. Then $0 \leq \cl(\chi) \leq m+1$ for any $\chi \in \Irr(G)$.
\item Let $q$ be any prime power, $A = \FF_q^n$ be endowed with a non-denegerate quadratic form as above, 
$\Omega(A) \lhd G \leq \GO(A)$, and $\chi \in \Irr(G)$. Then $0 \leq \cl(\chi) \leq \lfloor n/2 \rfloor+1$.
Moreover, $0 \leq \cl(\chi) \leq \lfloor n/2 \rfloor$ if $G \leq \SO(A)$ when $2 \nmid q$ and $G = \Omega(A)$ when $2|q$.
\end{enumerate}
\end{lemma}

\begin{proof}
For (i), we apply Lemma \ref{value} to the character $\Theta:= \om_n+\oms_n$ of $G = \Sp(A)$. By 
Proposition \ref{weil-sp}(ii), $\Theta^2 = 2(\tau_A+\zeta_A)$. It follows that $\Theta(g) = 0$ if $2 \nmid d(g)$, and 
$\Theta(g) = \pm 2q^{d(g)/2}$ if $2|d(g)$, where $d(g):=\dim\Ker(g-1_A)$ for any $g \in G$. Thus 
$$\Theta(g) \in \{0,\pm 2,\pm 2q, \ldots ,\pm 2q^{n-1},2q^n\},$$
and moreover $\Theta(g) = \Theta(1) = 2q^n$ if and only if $g = 1$. Hence the statement follows.

\smallskip
For (ii) and (iii), we apply Lemma \ref{value} to the character $\Theta:= \tau_A+\zeta_A$, where $A = \FF_q^n$ is the natural module
for $G$. Note that
$$\Theta(g) \in \{0,2,2q^2, 2q^4,\ldots ,2q^{2m-2},2q^{2m}\}$$
if $n = 2m$ is even, and 
$$\Theta(g) \in \{0,2q,2q^3, \ldots ,2q^{2m-1},2q^{2m+1}\}$$
if $n = 2m+1$ is odd. Moreover, $\Theta(g) = \Theta(1) = 2q^n$ if and only if $g = 1$.
Furthermore, if we assume $G \leq \SO(A)$ when $2 \nmid q$ and $G = \Om(A)$ when $2|q$, then 
$\Theta(g) = 2\tau_A(g) \neq 0$ by Lemma \ref{even}.
Hence the statements follow.
\end{proof}

Abusing the language, we will say that a character $\al$ of a finite group $G$ {\it contains} another character $\beta$ of 
$G$, if $\al-\beta$ is zero or a character of $G$.
Now we can prove the following three key lower bounds on the degree of any irreducible character of given level:

\begin{theor}\label{sp-bound1}
Let $G=G_n := \Sp_{2n}(q)$ with $q$ a fixed odd prime power. For any $k \in \ZZ_{\geq 1}$, define
$$\blc(n,k) := q^{nk-k(k+1)/2}\biggl(\frac{q-1}{2}\biggr)^k.$$
Suppose $\chi \in \Irr(G)$ and $\chi(1) < \blc(n,k)$. Then $\cl(\chi) \leq 3(k-1)$.
\end{theor}

\begin{proof}
We proceed by induction on $k$. If $k =1$, then 
$$\chi(1) < q^{n-1}(q-1)/2 \leq (q^n-1)/2,$$
hence $\chi = 1_G$ (see e.g. \cite[Theorem 1.1]{TZ1}) and $\cl(\chi) = 0$. 

For the induction step, we assume $k \geq 2$. If $n = 1$, then 
$$\chi(1) < \blc(1,k) < (q-1)/2,$$
again forcing $\chi = 1_G$ and $\cl(\chi) = 0$. So we will assume $n \geq 2$.
Consider the natural module $A = \FF_q^{2n}$ for $G$ so that $G = \Sp(A)$, and the stabilizer 
$P'=Q \rtimes H$ of a nonzero $v \in A$, with $Q$ a group of extraspecial type of order $q^{2n-1}$ and 
$H \cong G_{n-1}$. We may assume that $\chi \neq 1_G$ and $\chi$ is afforded by a $\CC G$-module $X$.
Since $\ZB(Q) \not\leq \ZB(G)$, 
the $\lam$-eigenspace $X_\lam$ of $\ZB(Q)$ on $X$ is nonzero for some nontrivial linear character $\lam$ of $\ZB(Q)$.
As shown in \cite[\S5]{GMST} and \cite[\S2]{MT}, the $P'$-module $X_\lam$ is isomorphic to $M_\lam \otimes Y$, where 
\begin{enumerate}[\rm(a)]
\item the $P'$-module $M_\lam$, with character say $\mu$, is irreducible over $Q$;
\item $Q$ acts trivially on the $P'$-module $Y$.
\end{enumerate} 
As $Q$ is normalized by $P:= \Stab_G(\langle v \rangle_{\FF_q})$ and the $P$-orbit of $\lam$ has length $(q-1)/2$, we see that
$$\dim Y \leq \frac{\dim X}{(\dim M_\lam)\cdot(q-1)/2} < \frac{\blc(n,k)}{q^{n-1}\cdot(q-1)/2} = \blc(n-1,k-1).$$
By the induction hypothesis, if $\psi$ is an irreducible constituent of the $H$-character afforded by $Y$, then $\cl(\psi) \leq 3(k-2)$.
Thus $\psi$ is contained in $(\om_{n-1}+\oms_{n-1})^{3k-6}$ by Definition \ref{def-level}(i).
Next, by \cite[Proposition 2.2(iii)]{TZ2} we have that $(\om_n+\oms_n)|_{P'}$ contains 
$\nu+\mu$, where $Q \leq \Ker(\nu)$ and $\nu|_H = \om_{n-1}+\oms_{n-1}$. It follows that 
$(\om_n+\oms_n)^{3k-5}|_{P'}$ contains $\mu\nu^{3k-6}$, which in turns contains $\mu\psi$. We have shown that
$$[\chi|_{P'},(\om_n+\oms_n)^{3k-5}|_{P'}]_{P'} > 0.$$ 
This implies by Frobenius' reciprocity that $\chi$ is contained in 
$$\Ind^G_{P'}\biggl(\bigl((\om_n+\oms_n)^{3k-5}\bigr)|_{P'} \biggr) = (\om_n+\oms_n)^{3k-5} \cdot \Ind^G_{P'}(1_{P'}).$$
Recalling that $P'=\Stab_G(v)$, we see that $\Ind^G_{P'}(1_{P'})$ is contained in the permutation character $\tau_A$ of $G$ 
acting on the point set of $A$, and the latter character is contained in $(\om_n+\oms_n)^2$ by Proposition \ref{weil-sp}(ii). Consequently,
$\chi$ is contained in $(\om_n+\oms_n)^{3k-3}$, i.e. $\cl(\chi) \leq 3k-3$, and the induction step is completed.
\end{proof}

\begin{theor}\label{sp-bound2}
Let $G:= \Sp_{2n}(q)$ with $q= 2^f$ and $n \geq 2$. For any $k \in \ZZ_{\geq 1}$, define
$$\blc'(n,k) := q^{2nk-k(2k+1)}\biggl(\frac{(q-1)^2}{2}\biggr)^k.$$
Suppose $\chi \in \Irr(G)$ and $\chi(1) < \blc'(n,k)$. Then $\cl(\chi) \leq 3(k-1)$.
\end{theor}

\begin{proof}
(i) We proceed by induction on $k$. If $k =1$ and $n \geq 2$, then 
$$\chi(1) < \blc'(n,1) \leq \frac{(q^n-1)(q^n-q)}{2(q+1)} = \dl(G),$$
by \cite[Theorem 1.1]{TZ1}, and so $\chi = 1_G$ and $\cl(\chi) = 0$. For the purposes of the inductive proof, we also
observe that $\blc'(1,1) = (q-1)^2/2q < q-1 = \dl(\Sp_2(q))$, and so the statement also holds for $n=1$.  

\smallskip
(ii) For the induction step, we assume $k \geq 2$. If $n = 2$, then 
$\cl(\chi) \leq 3 \leq 3(k-1)$ by Lemma \ref{level-range}(ii), and so we are done. Likewise, if 
$k \geq 3$ and $n \leq 5$, then $\cl(\chi) \leq 6 \leq 3(k-1)$ by Lemma \ref{level-range}(ii), and we are again done.
If $(q,k,n) = (2,2, \leq 5)$, then 
$$\chi(1) < \blc'(n,2) = 2^{4n-12},$$
and using \cite{GAP} one can check that $\cl(\chi) \leq 1$. If $(q,k,n) = (\geq 4,2,3)$, then
$$\chi(1) < \blc'(3,2) = q^2(q-1)^4/4,$$
and so $\cl(\chi) \leq 1$ by \cite[Theorem 6.1]{GT2}.
Hence, we may assume that 
\begin{equation}\label{for-n}
  n \geq 4,~(n,q) \neq (4,2),(5,2).
\end{equation}   

Consider the natural module $A = \FF_q^{2n}$ for $G$ so that $G = \Sp(A)$, 
with a symplectic basis 
$$(e_1,e_2, \ldots ,e_n,f_1, f_2, \ldots ,f_n),$$
and the stabilizer 
$P=P_2 :=\Stab_G(U) = Q \rtimes L$ of the isotropic subspace $U=\langle e_1,e_2 \rangle_{\FQ}$ of $A$, where 
$|Q| = q^{4n-5}$, $\ZB(Q) > [Q,Q]$, $\ZB(Q)$ elementary abelian of order $q^3$, $[Q,Q]$ elementary abelian of 
order $q$, and $L \cong \GL_2(q) \times \Sp_{2n-4}(q)$, see \cite[\S3]{GT2}. Moreover, each of the $q-1$ non-identity elements in
$[Q,Q]$ is conjugate to a (fixed) short-root element $\bolds \in G$, and $L$ acts on $\Irr(\ZB(Q))$ with four orbits:
$\{1_{\ZB(Q)}\}$, $\OC_1$ consisting of all $q^2-1$ nontrivial linear characters of $\ZB(Q)/[Q,Q]$, 
$\OC^+_2$ of length $q(q^2-1)/2$, and $\OC^-_2$ of length $q(q-1)^2/2$.

We may assume that $\chi \neq 1_G$ and $\chi$ is afforded by a $\CC G$-module $X$. Since $[Q,Q] \not\leq \ZB(G)$,
there must be some $\eps= \pm$ such that the $\lam$-eigenspace $X_\lam$ of $\ZB(Q)$ on $X$ is nonzero for some linear character 
$\lam \in \OC^\eps_2$ of $\ZB(Q)$. We also consider $P':=\Stab_G(e_1,e_2) = Q \rtimes H$ with $H \cong \Sp_{2n-4}(q)$.
In fact, we can write $U^\perp = U \oplus W$ such that $H = \Sp(W)$, where
$$W = \langle e_3,e_4, \ldots, e_n,f_3, f_4, \ldots ,f_n \rangle_{\FQ},$$
Recall the {\it linear-Weil characters} $\rho^1_n$, $\rho^2_n$, $\tau^i_n$, $1 \leq i \leq (q-2)/2$, and the 
{\it unitary-Weil characters} $\al_n$, $\beta_n$, $\zeta^j_n$,  $1 \leq j \leq q/2$, of $\Sp_{2n}(q)$, see \cite[Table I]{GT2}. In particular, 
$$\tau_A = 2 \cdot 1_G + \rho^1_n + \rho^2_n + 2\sum^{(q-2)/2}_{i=1}\tau^i_n,~
     \zeta_A = \al_n + \beta_n + 2\sum^{q/2}_{i=1}\zeta^i_n.$$
As shown in \cite[Lemma 9.2]{GT2}, there exists an irreducible $\CC G$-module $M$ of $G$ with the following properties:
\begin{enumerate}[\rm(a)]
\item $M$ affords 
the character $\rho^1_n$ of degree $(q^n+1)(q^n-q)/2(q-1)$ when $\eps = +$ and the character 
$\al_n$ of degree $(q^n-1)(q^n-q)/2(q+1)$ when $\eps = -$; and
\item the $\lam$-eigenspace $M_\lam$ of $\ZB(Q)$ in $M$ is a $P'$-module with character $\mu$ of degree $q^{2n-4}$, and
furthermore, $\mu|_Q$ is the unique irreducible character of $Q$ that lies above the character $\lam$ of $\ZB(Q)$. 
\end{enumerate}
It follows by Gallagher's theorem \cite[Corollary 6.17]{Is} that some irreducible constituent of the character of the 
$P'$-module $X_\lam$ can be written as $\mu\psi$, where $Q \leq \Ker(\psi)$ and $\psi|_H \in \Irr(H)$.
As $Q$ is normalized by $P$ and the $P$-orbit $\OC^\eps_2$ of $\lam$ has length at least $q(q-1)^2/2$, we see that
$$\psi(1) \leq \frac{\dim X}{\dim M_\lam \cdot |\OC^\eps_2|} < \frac{\blc'(n,k)}{q^{2n-4}\cdot q(q-1)^2/2} = \blc'(n-2,k-1).$$
By the induction hypothesis applied to the character $\psi|_H$ of $H = \Sp(W) \cong \Sp_{2n-4}(q)$, $\cl(\psi) \leq 3(k-2)$.
Thus $\psi|_H$ is contained in $(\tau_W+\zeta_W)^{3k-6}$ by Definition \ref{def-level}(ii). Also, as $\rho^1_n$ is contained in
$\tau_A$ and $\al_n$ is contained in $\zeta_A$, we also see that $\mu$ is contained in $(\tau_A+\zeta_A)|_{P'}$. 

\smallskip
(iii) In addition to $P=P_2$, we also consider the parabolic subgroup 
$$P_1 := \Stab_G(\langle e_1 \rangle_{\FQ} = Q_1 \rtimes L_1,$$
with Levi subgroup $L_1 = \Sp_{2n-2}(q) \times T_1$ fixing $\langle f_1 \rangle_{\FQ}$ and $T_1 \cong C_{q-1}$. 
Recall the assumption \eqref{for-n}, we then have by \cite[Proposition 7.4]{GT2} that
\begin{equation}\label{hc1}
  \SR^G_{L_1}(\al_n) = \al_{n-1} \otimes 1_{T_1},~\SR^G_{L_1}(\beta_n) = \beta_{n-1} \otimes 1_{T_1},~\SR^G_{L_1}(\tau^i_n) = 
    \tau^i_{n-1} \otimes 1_{T_1},
\end{equation}    
if $\SR^G_{L_1}$ denotes the Harish-Chandra restriction from $G$ to $L_1$. Applying the same statement to
$$K:= [L_1,L_1] = \Stab_G(e_1,f_1) \cong \Sp_{2n-2}(q),$$ 
we obtain
\begin{equation}\label{hc2}
  \SR^K_{L_{11}}(\al_{n-1}) = \al_{n-2} \otimes 1_{T_{11}},~\SR^K_{L_{11}}(\beta_{n-1}) = \beta_{n-2} \otimes 1_{T_{11}},
  ~\SR^K_{L_{11}}(\tau^i_{n-1}) = \tau^i_{n-2} \otimes 1_{T_{11}}.
\end{equation}      
Here, $L_{11} = \Sp_{2n_4}(q) \times T_{11} = \Stab_K(\langle f_2 \rangle_{\FQ}$ is a Levi subgroup of 
the parabolic subgroup $P_{11} := \Stab_K(\langle e_2 \rangle_{\FQ}) = Q_{11} \rtimes L_{11}$ of $K$ and $T_{11} \cong C_{q-1}$.

Consider any element $g \in Q$. Since $g(e_1) = e_1$, $g(e_2) = e_2$, and $g$ acts trivially on
$U^\perp/U$,  $g \in Q_1K$ and furthermore $g$ projects onto 
an element $h \in Q_{11}$ under $Q_K \to Q_1K/Q_1$. It follows that $g$ acts trivially on $\SR^K_{L_{11}}((\SR^G_{L_1}(Z))|_K)$ for any 
$G$-module $Z$. Hence, \eqref{hc1} and \eqref{hc2} imply that 
$(\SR^G_L(\al_n))|_H$ contains $\al_{n-2}$. Arguing similarly with $\beta_n$ and $\zeta^i_n$, we conclude that
$(\SR^G_L(\zeta_A))|_H$ contains $\zeta_W$. The same arguments, but using \cite[Proposition 7.9]{GT2} 
show that $(\SR^G_L(\tau_A))|_H$ contains $\tau_W$.

\smallskip
(iv) The results of (ii) and (iii) imply that 
$(\tau_A+\zeta_A)^{3k-5}|_{P'}$ contains $\mu\nu^{3k-6}$, which in turns contains $\mu\psi$. We have shown that
$$[\chi|_{P'},(\tau_A+\zeta_A)^{3k-5}|_{P'}]_{P'} > 0.$$ 
This implies by Frobenius' reciprocity that $\chi$ is contained in 
$$\Ind^G_{P'}\biggl(\bigl((\tau_A+\zeta_A)^{3k-5}\bigr)|_{P'} \biggr) = (\tau_A+\zeta_A)^{3k-5} \cdot \Ind^G_{P'}(1_{P'}).$$
Recalling that $P'=\Stab_G(e_1,e_2)$, we see that $\Ind^G_{P'}(1_{P'})$ is contained in the permutation character $(\tau_A)^2$ of $G$ 
acting on the point set of $A \times A$. Consequently,
$\chi$ is contained in $(\tau_A+\zeta_A)^{3k-3}$, i.e. $\cl(\chi) \leq 3k-3$, and the induction step is completed.
\end{proof}

\begin{theor}\label{so-bound}
Let $G:= \Omega^\eps_n(q)$ with $q = p^f$ a fixed power of a prime $p$, 
$\eps = \pm$, and $n \geq 6$. For any $k \in \ZZ_{\geq 1}$, define
$$\bld(n,k) := \left\{ \begin{array}{ll}q^{nk-2k(k+1)}(q-1)^k, & (n,k) \neq (8,2), (9,2),\\
                               q^4(q-1)^2/\gcd(2,q-1), & (n,k) = (8,2),\\
                               q^5(q^2-1)(q-1)/2, & (n,k) = (9,2). \end{array} \right.$$
Suppose $\chi \in \Irr(G)$ and $\chi(1) < \bld(n,k)$. Then $\cl(\chi) \leq 3(k-1)$.
\end{theor}

\begin{proof}
(i) We proceed by induction on $k$ and use the formula for the smallest degree $\dl(G)$ of nontrivial complex irreducible characters of
$G$ as determined in \cite[Theorem 1.1]{TZ1}. For the purposes of the inductive proof, we also define
$$\bld(4,1) = (q-1)/\gcd(2,q-1),~\bld(5,1) = (q^2-1)/2.$$
(Note that $n$ is assumed to be even when $2|q$.) Suppose $k =1$ and $n \geq 4$. Then
$$\chi(1) < \bld(n,1) \leq \dl(G).$$
If $(n,q,\eps) \neq (4,2)$, then $G$ is perfect, hence $\chi = 1_G$ and $\cl(\chi) = 0$. If $(n,q,\eps) = (4,2)$,
then $\bld(4,1) = 1$ and so the statement is also vacuously true.

\smallskip
(ii) For the induction step, we assume $k \geq 2$. If $n = 6$ or $7$, then
$\cl(\chi) \leq 3 \leq 3(k-1)$ by Lemma \ref{level-range}(ii), and so we may assume $n \geq 8$. Likewise, if $k \geq 3$ but $n \leq 13$, then
again $\cl(\chi) \leq 6 \leq 3(k-1)$ by Lemma \ref{level-range}(ii), and we are done. Thus we may assume that $n \geq 14$ when $k \geq 3$.
It follows for the $(n,k)$ in question that 
\begin{equation}\label{ratio1}
  \frac{\bld(n,k)}{q^{n-4}(q-1)} \leq \bld(n-4,k-1).
\end{equation}

Let $A:= \FF_q^{n}$ be the natural module for
$G = \Om(A)$, endowed with a non-degenerate quadratic form $\QF$, whose bilinear form has Gram matrix
$\diag\left(\left( \begin{array}{cc}0 & I_{2}\\I_{2} & 0 \end{array}\right), J \right)$ with respect to a basis 
$$(u_1,u_2,v_1,v_2,w_1 \ldots, w_{n-4})$$ 
(for a suitable invertible matrix $J$). We will assume that 
$\langle u_1, u_2 \rangle_{\FQ}$ is totally singular. Consider the parabolic subgroup
$P = \Stab_G(U)$ where $U :=\langle u_1,u_2 \rangle_{\FQ}$, and its subgroup $P':= \Stab_G(u_1,u_2)$. Then, as shown in 
\cite[\S\S6, 7]{MT}, $Q := \OB_p(P)$ is a $p$-group of order $q^{2n-7}$ of extraspecial type, and 
$P'=Q \rtimes H$, where 
$$H := \Stab_G(u_1,u_2,v_1,v_2) = \Om(W) \cong \Om^\eps_{n-4}(q),$$
where $W : = \langle w_1, w_2, \ldots,w_{n-4} \rangle_{\FQ}$ is a quadratic space of the same type $\eps$ as of $A$.
Note that $P$ acts transitively on the $q-1$ non-identity elements of the elementary abelian $p$-group $\ZB(Q)$, and also on 
its $q-1$ nontrivial linear characters. Fix some 
$\lam \in \Irr(\ZB(Q)) \smallsetminus \{1_{\ZB(Q)}\}$.  

Consider an irreducible $\CC G$-module $M$ of dimension 
$$\dim M = \left\{ \begin{array}{ll}(q^{2m}-1)/(q^2-1), & G \cong \Om_{2m+1}(q),\\
    (q^m -\eps)(q^{m-1}+\eps q)/(q^2-1), & G \cong \Om^\eps_{2m}(q). \end{array}\right.$$
As shown in \cite[Proposition 5.7]{LBST1}, when $2 \nmid qn$ this module $M$ can be taken to afford the character
$\theta=D^\circ_{1_{\Sp_2(q)}}$, which is an irreducible constituent of $\tau_A$. On the other hand, as seen in 
\cite[Table 1]{ST}, when $2|n$ this module $M$ can be taken with character $\theta$ an irreducible constituent of the 
rank $3$ permutation character of $G$ acting on the set of singular $1$-dimensional subspaces of $A$, which is again 
contained in $\tau_A$.

As $Q$ is of extraspecial type, $Q$ has a unique irreducible $\CC Q$-module, say $N_\lam$, with central character $\lam$, 
and $\dim N_\lam = q^{n-4}$. Since $\ZB(Q)$ does not act trivially on $M$, the
$\lam$-eigenspace $M_\lam$ of $\ZB(Q)$ on $M$ is nonzero. First we assume that $q \geq 3$. Then 
$$\dim M_\lam \leq \dl(G)/(q-1) < 2q^{n-4}.$$
It follows that $M_\lam$ is irreducible over $Q$, $(M_\lam)|_Q \cong N_\lam$. 
Consider the case $q=2$; in particular, $n=2m$. Then we embed $G = \Om(A) \cong \Om^\eps_{2m}(2)$ in 
$\tilde G := \Sp(A) \cong \Sp_{2m}(2)$, and consider $\tilde P := \Stab_G(U)$, with unipotent radical $\tilde Q$ of order 
$2^{4n-5}$. As shown in \cite[\S3]{GT2}, $[\tilde Q,\tilde Q] = \{1,\bolds\}$ can then be identified with $\ZB(Q)$, where 
$\bolds$ is a short-root element in $\tilde G$. By \cite[Lemma 5.13]{LBST1},
$$\theta+1_G = \left\{ \begin{array}{ll}(\beta_n)|_G, & \mbox{if }\eps = +,\\
    (\al_n)|_G, & \mbox{if }\eps= -,\end{array}\right.$$
where $\al_n,\beta_n \in \Irr(\tilde G)$ are {\it unitary-Weil characters} of $\tilde G$ of degree $(2^n-1)(2^{n-1}-1)/3$ and 
$(2^n+1)(2^{n-1}+1)/3$, respectively, cf. \cite[Table I]{GT2}. Now computing $\theta(\bolds)$ using \cite[Corollary 7.3]{GT2},   
we see that $\dim M_\lam = 2^{2m-4}$, and so we again have that $(M_\lam)|_Q = N_\lam$.  We have shown that 
$N_\lam$ extends to $P$ for all $q$.
Let $\mu$ denote the character of the $P'$-module $M_\lam$, and 
note that we have now shown that $(\tau_A)|_{P'}$ contains $\mu$.

\smallskip
(iii) We may assume that $\chi \neq 1_G$ and $\chi$ is afforded by a $\CC G$-module $X$. Then 
the $\lam$-eigenspace $X_\lam$ of $\ZB(Q)$ on $X$ is nonzero. The result of (ii) and Gallagher's theorem
\cite[Corollary 6.17]{Is} show that  the $P'$-module $X_\lam$ is isomorphic to $M_\lam \otimes Y$, where 
$Q$ acts trivially on the $P'$-module $Y$. Furthermore, as the $P$-orbit of $\lam$ has length $q-1$, we see by \eqref{ratio1} that
$$\dim Y \leq \frac{\dim X}{(\dim M_\lam)\cdot(q-1)} < \frac{\bld(n,k)}{q^{n-4}\cdot(q-1)} \leq \blc(n-4,k-1).$$
By the induction hypothesis applied to $H = \Om(W) \cong \Om^\eps_{n-4}(q)$ and $k-1 \geq 1$, if $\psi$ is an irreducible constituent of 
the $H$-character afforded by $Y$, then $\cl(\psi) \leq 3(k-2)$.
Thus $\psi$ is contained in $(\tau_W)^{3k-6}$ by Definition \ref{def-level}(ii).

Next, recall that $\tau_A$ is the character of the $\CC G$-module $R$ with basis $(e_a \mid a \in A)$, where $g \in G$ sends 
$e_a$ to $e_{g(a)}$. For any coset $\bar{v}=v+U$ in $U^\perp/U$, we define
$$e_{\bar{v}} := \sum_{x \in U}e_{v+x},$$
where the representative $v$ can be chosen (uniquely) from $W$.  
Then the $q^{n-4}$ vectors $e_{\bar{v}}$ with $\bar{v} \in U^\perp/U$ are linearly independent, and so they span a subspace 
$T \cong \CC^{q^{n-4}}$ of $R$. As $Q$ acts trivially on $U^\perp/U$, $Q$ also acts trivially on $T$. Furthermore, we can 
write $U^\perp = U \oplus W$, with $H = \Om(W)$ acting trivially on $U$ and naturally on $W$. Thus the $H$-module $T$ affords 
the $H$-character $\tau_W$. Since $(\tau_A)|_H = q^4\tau_W$, we have shown that $(\tau_A)|_{P'}$ contains 
$\nu$, where $Q \leq \Ker(\nu)$ and $\nu|_H = \tau_W$, whence $(\tau_A)^{3k-6}|_{P'}$ contains $\nu^{3k-6}$, which in 
turn contains $\psi$. 

On the other hand, it was shown in (ii) that $(\tau_A)|_{P'}$ contains $\mu$. It follows that
$(\tau_A)^{3k-5}|_{P'}$ contains $\mu\psi$. Hence,
$$[\chi|_{P'},(\tau_A)^{3k-5}|_{P'}]_{P'} > 0.$$ 
This implies by Frobenius' reciprocity that $\chi$ is contained in 
$$\Ind^G_{P'}\biggl(\bigl((\tau_A)^{3k-5}\bigr)|_{P'} \biggr) = (\tau_A)^{3k-5} \cdot \Ind^G_{P'}(1_{P'}).$$
Recalling that $P'=\Stab_G(u_1,u_2)$, we see that $\Ind^G_{P'}(1_{P'})$ is contained in the permutation character $(\tau_A)^2$ of $G$ 
acting on the point set of $A \times A$. Consequently,
$\chi$ is contained in $(\tau_A)^{3k-3}$, i.e. $\cl(\chi) \leq 3k-3$, and the induction step is completed.
\end{proof}

\begin{proof}[Proof of Theorem \ref{main3}]
By the choice of $k$, $\ell > 3(k-1)$. Hence the lower bound on $\chi(1)$ follows from Theorems \ref{sp-bound1}--\ref{so-bound}.
The upper bound in case (i)  follows from the fact that any irreducible constituent of $\om_n+\oms_n$ has degree at most 
$(q^n+1)/2$. Likewise, in cases (ii) and (iii), any irreducible constituent of $\tau_A$ over $\GL_n(q)$ has degree at most 
$(q^n-1)/(q-1)$ and any irreducible constituent of $\zeta_A$ over $\GU_n(q)$ has degree at most 
$(q^n+1)/(q+1)$, if $A = \FF_q^n$. 
\end{proof}

Note that Lemma \ref{level-range} only lists a possible range for the level of any irreducible character of classical groups. 
Next we prove the existence of characters of relatively small levels.

\begin{propo}\label{sp-dual}
In the situation of \eqref{pairs}{\rm (a)}, assume that $n \geq m(m-1)+3$.
Then for any $\al \in \Irr(S)$, $D_\al$ is a character of $G= \Sp_{2n}(q)$. Furthermore, if one defines 
$$D'_\al := \sum_{\psi \in \Irr(G),~\cl(\psi) \leq m-1}[D_\al,\psi]_G\,\psi,$$
then $\DC_\al := D_\al -D'_\al$ is a character of $G$, all of whose irreducible constituents 
have level $m$.
\end{propo}

\begin{proof}
The statements are obvious for $m=1$, so we will assume $m \geq 2$. 
For the first statement, we apply \eqref{dual1} to $g=1$. For any $1 \neq s \in S$, note that $gs$ as an element in $\Gamma$ has 
the eigenvalue $1$ with multiplicity at most $2n(m-1)$. It follows from Proposition \ref{weil-sp}(i) that 
$$|\om(gs)| \leq q^{n(m-1)}.$$
Also, note that 
\begin{equation}\label{S-order}
  |S| \leq (8/3)q^{m(m-1)/2}.
\end{equation}  
Hence, \eqref{dual1} implies that
\begin{equation}\label{D-degree}
  |D_\al(1)| > \frac{\al(1)}{|S|} \bigl(q^{nm} - q^{n(m-1)}|S| \bigr) = \frac{\al(1)q^{nm}}{|S|} \biggl(1 - \frac{|S|}{q^n} \biggr) 
    > \frac{9}{10} \cdot \frac{\al(1)q^{nm}}{|S|},
\end{equation}       
and so $D_\al$ is a (true) character of $G$.

To prove the second statement, it suffices to show that $D'_\al(1) < D_\al(1)$.
By Proposition \ref{weil-sp} and Lemma \ref{orbits2}, for any $0 \leq j \leq m$ we have that
$$[(\om_n)^j(\oms_n)^{m-j},(\om_n)^j(\oms_n)^{m-j}]_G = [(\om_n\overline{\om_n})^j(\oms_n\overline{\oms_n})^{m-j}]_G =
    [(\tau_A)^m,1_G]_G \leq 6q^{m(m-1)/2}.$$
Since $\om|_G$ is equal to some such $\om_n^j(\oms_n)^{n-j}$ by Proposition \ref{weil-sp}(iv), we obtain by Lemma \ref{mults}(i) that 
$$\Csum(\om,G) \leq [\om|_G,\om|_G]_G \leq 6q^{m(m-1)/2}.$$
On the other hand, if $k \leq m-1$ then any irreducible constituent of $(\om_n+\oms_n)^k$ is an irreducible constituent of some
$\om_n^{k-i}(\oms_n)^{i}$ with $0 \leq i \leq k$, and so has 
degree at most $q^{nk} \leq q^{n(m-1)}$. Hence, using \eqref{S-order} we get
$$D'_\al(1)|S| \leq 6q^{m(m-1)/2}q^{n(m-1)}(8/3)q^{m(m-1)/2} = 16q^{n(m-1)+m(m-1)} < (9/10)q^{nm}$$
since $n \geq m(m-1)+3$, and so we are done by \eqref{D-degree}. 
\end{proof}

\begin{propo}\label{so-dual}
In the situation of \eqref{pairs}{\rm (b)}, assume that $m \geq 2n(2n+1)+2$. 
Also, let $\tau_B$ denote the permutation character of $G = \SO(B) = \SO^{\pm}_m(q)$ acting on the point set of the 
natural module $B = \FF_q^m$. Then for any $\al \in \Irr(S)$, $D_\al$ is a character of $G$. Furthermore, if one defines 
$$D'_\al := \sum_{\psi \in \Irr(G),~\cl(\psi) \leq n-1}[D_\al,\psi]_G\,\psi,$$
then $\DC_\al := D_\al -D'_\al$ is a character of $G$, all of whose irreducible constituents 
have level $n$.
\end{propo}

\begin{proof}
We again apply \eqref{dual1} to $g=1$. For any $1 \neq s \in S$, note that $gs$ as an element in $\Gamma$ has 
the eigenvalue $1$ with multiplicity at most $m(2n-1)$, hence by Proposition \ref{weil-sp}(i) we have that 
$$|\om(gs)| \leq q^{m(n-1/2)}.$$
Also, note that 
\begin{equation}\label{S-order2}
  |S| < q^{n(2n+1)}.
\end{equation}  
Hence, \eqref{dual1} implies that
\begin{equation}\label{D-degree2}
  |D_\al(1)| > \frac{\al(1)}{|S|} \bigl(q^{nm} - q^{m(n-1/2)}|S| \bigr) = \frac{\al(1)q^{nm}}{|S|} \biggl(1 - \frac{|S|}{q^{m/2}} \biggr) 
    > \frac{26}{27} \cdot \frac{\al(1)q^{nm}}{|S|},
\end{equation}       
and so $D_\al$ is a (true) character of $G$.

To prove the second statement, it again suffices to show that $D'_\al(1) < D_\al(1)$.
By Lemma \ref{orbits2} we have that
$$[(\tau_B)^n,(\tau_B)^n]_G = [(\tau_B)^{2n},1_G]_G \leq 6q^{n(2n+1)}.$$
Since $\om|_G = (\tau_A)^n$ by Proposition \ref{weil-sp}(iii), we obtain by Lemma \ref{mults}(i) that 
$$\Csum(\om,G) \leq [\om|_G,\om|_G]_G \leq 6q^{n(2n+1)}.$$
On the other hand, if $k \leq n-1$ then any irreducible constituent of $(\tau_B)^k$ has 
degree at most $q^{mk} \leq q^{m(n-1)}$. Hence, using \eqref{S-order2} we get
$$D'_\al(1)|S| \leq 6q^{n(2n+1)}q^{m(n-1)}q^{n(2m+1)} = 6q^{m(n-1)+2n(2n+1)} < (26/27)q^{nm}$$
since $m \geq n(2n+1)+2$, and so we are done by \eqref{D-degree2}. 
\end{proof}

For odd primes $p$, a strengthening of Proposition \ref{so-dual} is given in Corollary \ref{so-dual2} (see below).

\section{Character level and the $U$-rank}\label{rank-classical}
A local approach to stratify irreducible characters of finite classical groups $G$ is via the study of their restriction to ``nice'' subgroups, such
as parabolic subgroups. This approach has been used in \cite{GMST}, \cite{GT2}, \cite{T2}, \cite{T3}. More recently, this approach has been 
taken in \cite{GH1}, \cite{GH2} to develop the $U$-rank theory. In this section, we prove some results concerning the $U$-rank. Since these 
results will not be used in the subsequent sections of the paper, we will restrict ourselves to the case of orthogonal groups
in odd characteristic.
 
Let $G = \SO(B) \cong \SO^+_{2n}(q)$, where $q$ is a power of a prime $p > 2$ and 
$$B = \langle u_1, \ldots ,u_n,v_1, \ldots,v_n \rangle_{\FQ}$$ 
is endowed with 
a non-degenerate symmetric bilinear form with Gram matrix $\begin{pmatrix}0 & I_n\\I_n & 0 \end{pmatrix}$
in the given basis. Fix a primitive $p^{\mathrm {th}}$ root $\eps$ of unity. 
Consider the subspace $W_j = \langle u_1, \ldots ,u_j \rangle_{\FQ}$ for any $1 \leq j \leq n$ and its stabilizer 
$P_j = U_j \rtimes L_j$ with unipotent radical $U_j$ and
\begin{equation}\label{center-uj}
  \ZB(U_j) = \biggl\{ [I_j,X] := \begin{pmatrix}I_j & 0 & X \\ 0 & I_{2n-2j} & 0\\ 0 & 0 & I_j \end{pmatrix} \mid X \in M_{j,j}(\FQ), X + \tw tX = 0 \biggr\}.
\end{equation}  
Then any character $\lam \in \Irr(\ZB(U_j))$ can be written uniquely in the form
$$\lam=\lam_Y : [I_j,X] \mapsto \eps^{\Tr_{\FQ/\FF_p} \tr(XY)}$$
for some $Y \in M_{j,j}(\FQ)$ with $Y + \tw t Y = 0$, and the {\it rank} of $\lam_Y$ is defined to be $\sfr(\lam_Y) := \rank(Y)$ (which in this 
case is always an even number). 

\begin{defi}{\rm \cite[\S4]{GH2}}\label{def:rank}
{\em \begin{enumerate}[\rm(i)]
\item For any complex character $\chi$ of $\SO^+_{2n}(q)$, the {\it $U$-rank} of $\chi$, $\sfr(\chi)$, is defined to be the largest among all the ranks 
$\sfr(\nu)$, where $\nu$ is any irreducible constituent of $\chi|_{\ZB(U_j)}$ and $1 \leq j \leq n$. Similarly, if $\psi$ is any complex character of $P_j$, 
the {\it $U$-rank} $\sfr(\psi)$ is the largest among all the ranks 
$\sfr(\nu)$, where $\nu$ is any irreducible constituent of $\psi|_{\ZB(U_j)}$.
\item Let $H :=  \SO(\tilde B) = \SO_{2n+1}(q)$ or $\SO^-_{2n+2}(q)$. Embed $G = \SO^+_{2n}(q)$ in $H$, as the point-wise stabilizer of a $1$-dimensional,
respectively $2$-dimensional, non-degenerate subspace 
in the natural module $\tilde B$ for $H$. Then for any character $\chi$ of $G$, the {\it $U$-rank} of $\chi$,
$\sfr(\chi)$, is defined to be $\sfr(\chi|_G)$.
\end{enumerate}} 
\end{defi}

In what follows, a {\it standard subgroup $\SO^+_{2m}(q)$} with $1 \leq m \leq n$ 
means (any conjugate in $G$ of) the pointwise stabilizer 
$\Stab_G(u_{m+1}, \ldots ,u_n,v_{m+1}, \ldots ,v_n)$. First we record some elementary properties of $\sfr(\chi)$.

\begin{lemma}\label{so-rank1}
With the above introduced notation, the following statements hold.
\begin{enumerate}[\rm(i)]
\item Suppose that $l=\sfr(\chi)$ for some $\chi \in \Irr(\SO^+_{2n}(q))$. If $H = \SO^+_{2l}(q)$ is a standard subgroup of $G$, 
then $\sfr(\chi|_H) = l$. Similarly, if $P \in \{P_l,P_n\}$ is a parabolic 
subgroup of $G$, then $\sfr(\chi|_P) = l$. 
\item Suppose that $\sfr(\chi) = l$ and $\sfr(\theta) = k$ for $\SO^+_{2n}(q)$-characters $\chi,\theta$ and that $k+l \leq n$. Then 
$\sfr(\chi\theta) = k+l$.
\end{enumerate}
\end{lemma}

\begin{proof}
(i) First we prove the statement for $H$. 
By the definition, there is some $i$ with $l \leq i \leq n$ such that $\chi|_{\ZB(U_i)}$ contains an irreducible constituent $\lam=\lam_Y$ with 
$\rank(Y) = l =\sfr(\chi)$. Note that $L_i$ acts on the constituents of $\chi|_{\ZB(U_i)}$ via conjugation, and conjugating $\lam$ by a suitable element
in $L_i$, we may assume that $Y = \begin{pmatrix}0 & I_{l/2} & 0\\-I_{l/2} & 0 & 0\\ 0 & 0 & 0 \end{pmatrix}$. Now we take $H$ to be the pointwise 
stabilizer 
$$\Stab_G(u_{l+1}, \ldots, u_n,v_{l+1}, \ldots, v_n),$$ and consider the subspace
$W_l:=\langle u_1, u_2, \ldots ,u_l \rangle_{\FQ}$
and its stabilizer $P_H$ in $H$ with abelian unipotent radical $U_H$. Then it is easy to see that $U_H \leq \ZB(U_i)$ and 
$\lam|_{U_H} = \lam_Z$, with $Z = \begin{pmatrix}0 & I_{l/2}\\-I_{l/2} & 0 \end{pmatrix}$.
As $\rank(Z) = l$, we have $\sfr(\chi|_H) \geq l$. On the other hand, as $U_H \leq \ZB(U_i)$, we have that 
$\sfr(\chi|_H) \leq \sfr(\chi) = l$, and so $\sfr(\chi|_H) = l$.

Next, consider $P_l := \Stab_G(W_l)$ with unipotent radical $U_l$, and observe that $\ZB(U_l)$ can be identified with $U_H$. 
As $\rank(Z) = l$, we have $\sfr(\chi|_{P_l}) \geq l$. On the other hand,
$\sfr(\chi|_{P_l}) \leq \sfr(\chi) = l$ by Definition \ref{def:rank}, hence $\sfr(\chi|_{P_l}) = l$.

Finally, consider $P_n := \Stab_G\bigl(\langle u_1, u_2, \ldots ,u_n \rangle_{\FQ}\bigr)$ with unipotent radical $U_n$. Since 
$U_H \leq U_n$ and $\rank(Z) = l$, we have $\sfr(\mu) \geq l$ for any $\mu \in \Irr(U_n)$ lying above $\lam_Z$, and so 
$\sfr(\chi|_{P_n}) \geq l$. It then follows from Definition \ref{def:rank} that $\sfr(\chi|_{P_n}) = l$.

\smallskip
(ii) The arguments in (i) show that $\chi|_{U_n}$ contains an irreducible constituent $\lam = \lam_X$ of rank $l$ and 
$\theta|_{U_n}$ contains an irreducible constituent $\mu = \lam_Y$ of rank $k$. Again conjugating $\lam$ and $\mu$ by a suitable element in 
$P_n$, we may assume that 
$$Y = \begin{pmatrix}0 & I_{l/2} & 0\\-I_{l/2} & 0 & 0\\ 0 & 0 & 0\end{pmatrix},~Z = \begin{pmatrix}0 & 0 & 0\\0 & 0 & I_{k/2}\\0 & -I_{k/2} & 0\end{pmatrix}.$$
It follows that $(\chi\theta)|_{U_n}$ contains $\lam\mu = \lam_{Y+Z}$ with $\rank(Y+Z) = k+l$. Since the upper bound 
$\sfr(\chi\theta) \leq k+l$ is obvious, the statement follows.
\end{proof}

Recall the character $\tau=\tau_B$ from \eqref{tau1}.

\begin{corol}\label{so-rank2}
For any $\chi \in \Irr(\SO^+_{2n}(q))$, $2|\sfr(\chi) \leq \min(2\cl(\chi),n)$. The same statement holds for irreducible characters of 
$\SO(\tilde B) = \SO_{2n+1}(q)$, $\SO^-_{2n+2}(q)$.
\end{corol}

\begin{proof}
First we prove the statement for $G = \SO(B) = \SO^+_{2n}(q)$. Then 
the claim $2|\sfr(\chi) \leq n$ is immediate by the definition. By Definition \ref{def-level} and Lemma \ref{even},
$\chi$ is a constituent of $(\tau_B)^j$. Arguing similarly as in \cite[\S6.1]{ST}, one can show that $\sfr(\tau_B) = 2$. It then follows by Lemma \ref{so-rank1}(ii) that $\sfr(\chi) \leq 2j$.  

Next consider $H = \SO(\tilde B)$, where $\tilde B = \FF_q^N$ is a non-degenerate quadratic space, either of dimension $N = 2n+1$, or 
of type $-$ of dimension $N = 2n+2$. Then we can embed $G$ as the point-wise stabilizer of an $(N-2n)$-dimensional non-degenerate subspace
of $\tilde B$. Since $(\tau_{\tilde B})|_G = q^{N-2n}\tau_B$, it follows that every irreducible constituent $\chi$ of the restriction of $\varphi \in \Irr(H)$
to $B$ has level $\leq \cl(\varphi)$. Hence the statement for $\varphi$ follows from the statement for all $\chi$ in $\varphi|_G$.
\end{proof}

\begin{lemma}\label{so-upper}
Let $G = \SO(B) \cong \SO^+_{2n}(q)$ as above, and let $j=2r$ with $1 \leq r \leq n/2$. If $\lam \in \Irr(\ZB(U_j))$ has rank $2r$, where $U_j$ is the unipotent
radical of the $j^{\mathrm {th}}$ parabolic subgroup $P_j$ of $G$, then the multiplicity of $\lam$ in $(\tau_B)^r|_{\ZB(U_j)}$ is 
$q^{2r(n-2r)}|\Sp_{2r}(q)|$.
\end{lemma}

\begin{proof}
For any element $[I_j,X] \in \ZB(U_j)$, as given in \eqref{center-uj}, we have $\tau([I_j,X])= q^{2n-2r+2i}$, where $0 \leq i \leq r$. Hence, if we 
define
$$\Sigma:= \prod^{r-1}_{i=0}(\tau-q^{2n-2r+2i} \cdot 1_G),$$
then $\Sigma([I_j,X])$ is zero if $X \neq 0$, and $q^{2r(n-2r)}|\ZB(U_j)| \cdot |\Sp_{2r}(q)|$ if $X = 0$, i.e.
$$\Sigma|_{\ZB(U_j)} = q^{2r(n-2r)}|\Sp_{2r}(q)| \cdot \reg_{\ZB(U_j)}.$$
In particular, the multiplicity of $\lam$ in $\Sigma|_{\ZB(U_j)}$ is $q^{2r(n-2r)}|\Sp_{2r}(q)|$. Note that every irreducible constituent of 
$\tau^r-\Sigma$ has level $< r$, and so $\sfr(\tau^r-\Sigma) < 2r$ by Corollary \ref{so-rank2}. It follows
that $\lam$ is not a constituent of $(\tau^r-\Sigma)|_{\ZB(U_j)}$, and we are done.
\end{proof}

Note that each character $\lam$ of $\ZB(U_j)$ of rank $r$ as in Lemma \ref{so-upper} gives rise to an irreducible character $\psi_\lam$ of $U_j$ of degree 
$q^{2n(n-2r)}$. Hence the term $|\Sp_{2r}(q)|$ in the multiplicity of $\lam$ in $(\tau_B)^r|_{\ZB(U_j)}$ suggests that the $\psi_\lam$-homogeneous 
component of $(\tau_B)^r$ may carry the structure of a regular $\Sp_{2r}(q)$-module. This is clarified in the next statement. 

\begin{propo}\label{so-regular}
Let $q$ be any odd prime, and consider the dual pairs $G \times S \to \Gamma \cong \Sp_{4nr}(q)$ in {\rm \eqref{pairs}(b)}, where $G = \SO(B) \cong \SO^+_{2n}(q)$ and 
$S = \Sp(A) \cong \Sp_{2r}(q)$, with $1 \leq r \leq n/2$. Also consider the Siegel 
parabolic subgroup $P = P_n = \Stab_G(\langle u_1, \ldots, u_n\rangle_{\FQ})$ and 
its radical $U=U_n$. Let $\lam \in \Irr(U)$ be any character of rank $2r$. Then the restriction of $\om|_{G \times S}$ to $U \times S$ contains 
the character $\lam \otimes \reg_S$.
\end{propo}

\begin{proof}
Fix a Witt basis $(e_1, \ldots,e_r,f_1, \ldots,f_r)$ of the natural module $A = \FF_q^{2r}$ for $S = \Sp(A)$. Also consider a pair of 
complementary maximal totally isotropic subspaces in $V = B \otimes_{\FQ} A$:
$$W := \langle -u_i \otimes f_j,u_i \otimes e_j \rangle_{\FQ},~~W' := \langle v_i \otimes e_j,v_i \otimes f_j \rangle_{\FQ}.$$ 
Recall that $\om$ is the character of a reducible Weil module of dimension $q^{2nr}$ of $\Gamma = \Sp(V)$. We will fix the nontrivial
character 
$$\psi: (\FQ,+) \to \CC^\times,~x \mapsto \eps^{\Tr_{\FQ/\FF_p}(x)},$$ 
and use the model given in \cite[\S13]{Gr}
for such a representation, with $\Gamma$ acting on the space $\cW$ of complex-valued functions on $W'$. If $\delta_u$ denotes the delta-function
for any point $u \in W'$, then $(\delta_u \mid u \in W')$ is a basis for $\cW$. The action of the Siegel parabolic subgroup 
$\Stab_{\Gamma}(W)$ in this basis is described in \cite[(13.3)]{Gr}. 

Now fix the following vector
$$w:= (v_1 \otimes e_1 + v_2 \otimes f_1) + (v_3 \otimes e_2 + v_4 \otimes f_2) + \ldots + (v_{2r-1} \otimes e_r + v_{2r} \otimes f_r).$$
First we consider the action of any $g = [I,X] \in U$, given in the form of \eqref{center-uj} (with $j = n$). Using \cite[(13.3)]{Gr} we have
$$g(\delta_{v_1 \otimes e_1+v_2 \otimes f_1}) = \psi\biggl(\frac{1}{2}\bigl((g-1)(v_1 \otimes e_1+v_2 \otimes f_1),v_1 \otimes e_1+v_2 \otimes f_1\bigr)\biggr) =
    \psi(-x_{12})\delta_{v_1 \otimes e_1+v_2 \otimes f_1},$$ 
where $X = (x_{ij})$ is anti-symmetric. The same computation shows that 
$$g(\delta_w) = \psi(-x_{12}-x_{34} - \ldots -x_{2r-1,2r})\delta_w.$$
Choosing 
$$Y := \diag\biggl( \underbrace{\begin{pmatrix}0 & 1\\-1 & 0 \end{pmatrix}, \begin{pmatrix}0 & 1\\-1 & 0 \end{pmatrix}, \ldots, \begin{pmatrix}0 & 1\\-1 & 0     
     \end{pmatrix}}_{r~\mathrm{times}},  \underbrace{0, \ldots ,0}_{n-2r~\mathrm{times}}\biggr),$$
we then have 
$$g(\delta_w) = \psi(\tr(XY))\delta_w = \lam_Y(g)\delta_w,$$
for all $g \in U$, and with $\sfr(\lam_Y) = 2r$.

On the other hand, if $s \in S = \Sp(A)$, then $s$ fixes both $W$ and $W'$, and moreover $s(w) = w$ if and only if $s = 1$. As the action of 
$U \times S$ on $\cW$ is monomial in the basis $(\delta_u \mid u \in W')$, we conclude that 
$$\cW' := \langle \delta_{s(w)} \mid s \in S \rangle_\CC$$
is a $U \times S$-module, with character $\Ind^{U \times S}_U(\lam_Y \otimes 1_S) = \lam_Y \otimes \reg_S$.  Since $P$ acts transitively on
the set of $U$-characters of any given rank, the statement follows.
\end{proof}

Now we can prove the promised strengthening of Proposition \ref{so-dual}, which, in particular, establishes the existence of characters of 
$\SO^+_{2n}(q)$ of any level up to $n/2$ (when $q$ is odd):

\begin{corol}\label{so-dual2}
Let $q$ be any odd prime. 
Consider the dual pair $\tilde G \times S \to \Gamma \cong \Sp_{2Nr}(q)$ in {\rm \eqref{pairs}(b)}, where
$S = \Sp(A) \cong \Sp_{2r}(q)$, $1 \leq r \leq n/2$, and either 
\begin{enumerate}[\rm (a)] 
\item $N = 2n$ and $\tilde G = \SO(B) \cong \SO^+_{2n}(q)$, or 
\item $N = 2n+1$ and $\tilde G = \SO(\tilde B) = \SO_{2n+1}(q)$, or
\item $N = 2n+2$ and  $\tilde G = \SO(\tilde B) = \SO^-_{2n+2}(q)$.
\end{enumerate}
Then for any $\al \in \Irr(S)$, $D_\al$ is a character of $\tilde G$. Furthermore, if one defines 
$$D'_\al := \sum_{\nu \in \Irr({\tilde G}),~\cl(\nu) \leq r-1}[D_\al,\nu]_{\tilde G}\,\nu,$$
then $\DC_\al := D_\al -D'_\al$ is a character of $\tilde G$, all of whose irreducible constituents 
have level $r$. In fact, at least one irreducible constituent of $\DC_\al$ has $U$-rank $2r$.
\end{corol}

\begin{proof}
By Proposition \ref{weil-sp}(iii) and Corollary \ref{so-rank2}, for any irreducible constituent $\nu$ of $\om|_{\tilde G}$ we have that
\begin{equation}\label{2pairs0}
 \cl(\nu) \leq r,~\sfr(\nu) \leq 2r. 
\end{equation}  

First we consider the case of (a), so that $\tilde G = G$, and
consider the unipotent radical $U$ of a Siegel parabolic subgroup $P$ of $G$ and fix $\lam \in \Irr(U)$ of rank $2r$ as
in Proposition \ref{so-regular}. As $\om|_{U \times S}$ contains $\lam \otimes \reg_S$, it follows that $D_\al$ is a (true) character of $G$, and at least 
one irreducible constituent $\chi$ of $D_\al$ must afford $\lam$ on restriction to $U$. It follows that $\sfr(\chi) \geq \sfr(\lam) =2r$, whence
$\sfr(\chi) = 2r$ by \eqref{2pairs0} and $\cl(\chi) = r$ by Corollary \ref{so-rank2}. 
This in turn implies that $\chi$ is a constituent of $\DC_\al$, and so $\DC_\al$ is a (true) character of $G$.

Now we consider the cases of (b) and (c). To distinguish between the dual pair $\tilde G \times S$ in these cases and the one in (a),
we will use the notations with $\tilde{\;}$ for the objects in (b) or (c), and without for the ones in (a), with $G$ embedded in $\tilde G$ as in
the proof of Corollary \ref{so-rank2}. In particular, $\tilde\om = \om_{Nr}$ and 
\begin{equation}\label{2pairs1}
  \tilde\om|_{\tilde G \times S} = \sum_{\al \in \Irr(S)}\tilde D_\al \otimes \al;
\end{equation}   
respectively, $\om = \om_{2nr}$, and $\om|_{G \times S} = \sum_{\al \in \Irr(S)}D_\al \otimes \al$. Given the embedding $G \hookrightarrow \tilde G$ and 
the construction of the Weil representations affording $\om$ and $\tilde\om$,
it is not difficult to show that
\begin{equation}\label{2pairs2}
  \tilde\om|_{G \times S} = \bigl(1_G \otimes \om_r^{N-2n}\bigr) \cdot \om|_{G \times S} =  \sum_{\al \in \Irr(S)}D_\al \otimes \bigl( \om_r^{N-2n}\al \bigr).
\end{equation}    
Now, for any given $\al \in \Irr(S)$, we choose $\beta \in \Irr(S)$ so that 
$$0 < [\om_r^{N-2n}\overline\al,\overline\beta]_S = [\om_r^{N-2n}\beta,\al]_S.$$  
By the result for (a), we can find an irreducible constituent $\chi \in \Irr(G)$ of $D_\beta$ of level $r$ and $U$-rank $2r$. Then \eqref{2pairs2} implies
that $\chi \otimes \al$ is an irreducible constituent of $\tilde\om|_{G \times S}$. Restricting $\tilde\om$ down to $G \times S$ using \eqref{2pairs1}, we see
that there exists an irreducible constituent $\varphi \in \Irr(\tilde G)$ of $\tilde D_\al$ such that $\varphi|_G$ contains $\chi$. As $\sfr(\chi) = 2r$, 
it then follows from \eqref{2pairs0} that $\sfr(\varphi) = 2r$. Again using Corollary \ref{so-rank2} and \eqref{2pairs0}, we get $\cl(\varphi) = r$. This 
shows that both $\tilde D_\al$ and $\tilde{D}^\circ_\al$ are (true) characters of $\tilde G$ with $\varphi$ as an irreducible constituent.
\end{proof}

\section{Restrictions to natural subgroups}

In this section, by a {\it standard} subgroup $\GL_m(q)$ of $\GL_n(q) = \GL(V)$ we mean the subgroup
$$\Stab_{\GL(V)}(\langle e_1, \ldots,e_m \rangle_{\FF_q},e_{m+1}, \ldots,e_n) \cong \GL_m(q)$$
for some basis $(e_1, \ldots,e_n)$ of $V = \FF_q^n$. Likewise, by a {\it standard} subgroup $\GU_m(q)$ of $\GU_n(q) = \GU(V)$ we mean the subgroup
$$\Stab_{\GU(V)}(\langle e_1, \ldots,e_m \rangle_{\FF_{q^2}},e_{m+1}, \ldots,e_n) \cong \GU_m(q)$$
for some orthonormal basis $(e_1, \ldots,e_n)$ of the Hermitian space $V = \FF_{q^2}^n$. We will also use the notation 
$\GL^\eps_n(q)$ to denote $\GL_n(q)$ when $\eps = +$ and $\GU_n(q)$ when $\eps = -$.

\begin{propo}\label{tensor}
Let $G = \GL^\eps_n(q)$ with $n \geq 7$, $\eps = \pm$, and let 
$$\chi_1, \ldots,\chi_m,\chi'_1, \ldots, \chi'_m \in \Irr(G)$$ 
be of degrees at most $q^{nL}$ 
with $0 \leq L \leq n/5$. Then the following statements hold.
\begin{enumerate}[\rm(i)]
\item If $1 \leq r \leq n- \lfloor 1.4L \rfloor$ and $\eps= +$,
then the restriction of $\chi_i$ to a standard subgroup $\GL^\eps_r(q)$ contains a linear character of $\GL^\eps_r(q)$. 
\item In general, if $2.8mL \leq n$, then $[\chi_1\chi_2 \ldots \chi_m,\chi'_1\chi'_2 \ldots \chi'_m]_G \leq 8q^{2m^2L^2}$.
\end{enumerate}
\end{propo}

\begin{proof}
(a) Since $n \geq 7$, we have that 
$$\chi_i(1) \leq q^{n^2/5} < q^{n^2/4-2}.$$
It follows by \cite[Theorem 1.1]{GLT} that $j:= \cl(\chi_i) < n/2$ and so 
\begin{equation}\label{deg-1}
  q^{nL} \geq \chi_i(1) \geq q^{j(n-j)}.
\end{equation}  
Consider the polynomial $f(t) = t^2-nt+nL \in \CC[t]$. Then $f(1.4L) \leq 0$ and 
$$\frac{n+\sqrt{n^2-4nL}}{2} \geq \frac{n+L\sqrt{5}}{2} \geq n/2 +1.1L.$$
It follows that 
\begin{equation}\label{root1}
\mbox{If }t \leq n/2+1.1L \mbox{ and }f(t) \geq 0, \mbox{ then }t \leq 1.4L.
\end{equation}
In particular, \eqref{deg-1} implies
that $j \leq k := \lfloor 1.4L \rfloor$. 

\smallskip
(b) Here we assume that $\eps = +$. By the definition of the level $\cl(\chi_i)$ \cite[Definition 3.1]{GLT}, there is some 
linear character $\al_i \in \Irr(G)$ that $\chi_i\al_i$ is an irreducible constituent of $\tau^j$, where 
$\tau$ is the permutation character of $G$ acting on the set of vectors of the natural module $V = \FF_q^n$. 
Propositions 3.3, 3.5, and Theorem 3.6 of \cite{GLT} now imply that (i) holds.
Furthermore, note that $1_G$ is a constituent of $\tau$, so $\chi_i\al_i$ is also a constituent of $\tau^k$.
This is true for all $1 \leq i \leq m$. It follows that 
$$[\chi_1\chi_2 \ldots \chi_m,\chi_1\chi_2 \ldots \chi_m]_G$$ 
is at most $[\tau^{km},\tau^{km}]_G$, which equals the number of $G$-orbits on the set $\Omega_{2km}$ (in the notation of 
Lemma \ref{orbits1}). By Lemma \ref{orbits1} the latter is at most
$$8q^{k^2m^2} \leq 8q^{2m^2L^2}$$
since $2km \leq 2.8mL \leq n$. We have shown that if 
$$\chi_1\chi_2 \ldots \chi_m = \sum_{\gam \in \Irr(G)}a_\gam \gam,~~
    \chi'_1\chi'_2 \ldots \chi'_m = \sum_{\gam \in \Irr(G)}b_\gam \gam,$$
then 
$$\sum_{\gam \in \Irr(G)}a_\gam^2 \leq 8q^{2m^2L^2},~~\sum_{\gam \in \Irr(G)}b_\gam^2 \leq 8q^{2m^2L^2}.$$
It follows by the Cauchy--Schwarz inequality that
$$[\chi_1\chi_2 \ldots \chi_m,\chi'_1\chi'_2 \ldots \chi'_m]_G = \sum_{\gam \in \Irr(G)}a_\gam b_\gam \leq 
   \biggl(\sum_{\gam \in \Irr(G)}a_\gam^2 \cdot\sum_{\gam \in \Irr(G)}b_\gam^2\biggl)^{1/2} \leq 8q^{2m^2L^2}.$$
   
\smallskip
(c) Now we assume that $\eps = -$. By the definition of the level $\cl(\chi_i)$ \cite[Definition 4.2]{GLT}, there is some 
linear character $\al_i \in \Irr(G)$ that $\chi_i\al_i$ is an irreducible constituent of $\zeta^j$, where 
$\zeta = \zeta_n$ is defined in \cite[(4.1)]{GLT} and 
$\zeta^2$ is the permutation character of $G$ acting on the set of vectors of the natural module $V = \FF_{q^2}^n$. 
This is true for all $1 \leq i \leq m$. It follows that 
$$[\chi_1\chi_2 \ldots \chi_m,\chi_1\chi_2 \ldots \chi_m]_G$$ 
is at most $[\zeta^a,\zeta^a]_G$ for some $a \leq km$. Note that $[\zeta^a,\zeta^a]_G$  
equals the number of $G$-orbits on the set $\Omega_{km}$ (in the notation of 
Lemma \ref{orbits1}). By Lemma \ref{orbits1} the latter is at most
$$2q^{k^2m^2} \leq 2q^{2m^2L^2}$$
since $2km \leq 2.8mL \leq n$. Now we can finish as in (b).
\end{proof}

\begin{corol}\label{tensor2}
Let $G = \GL^\eps_n(q)$ with $n \geq 7$, $\eps = \pm$, and let $\chi_1, \chi_2$ be complex characters of $G$ of degree at most $q^{nL_1}$  and $q^{nL_2}$ 
with $L_1,L_2 \geq 0$ and $L_1 + L_2 \geq 1$. Then the following statements hold.
\begin{enumerate}[\rm(i)]
\item If $\chi_1, \chi_2 \in \Irr(G)$ then $\Csum(\chi_1\chi_2,G) \leq q^{5(L_1+L_2)^2}$.
\item In general, $\Csum(\chi_1\chi_2,G) \leq \Csum(\chi_1,G)\Csum(\chi_2,G)q^{5(L_1+L_2)^2}$. 
\end{enumerate}
\end{corol}

\begin{proof}
(i) Let $L := L_1+L_2$ and consider first the case $L \leq n/5$. 

Assume that $\eps = +$. Then the proof of Proposition \ref{tensor} shows that, for 
each $i = 1,2$, there exist a linear character $\al_i$ of $G$ such that $\chi_i\al_i$ is a constituent of $\tau^{k_i}$ for 
$k_i:= \lfloor 1.4L_2 \rfloor$. Hence $\chi_1\chi_2\al_1\al_2$ is a constituent of $\tau^k$ with 
$$k := \lfloor 1.4L \rfloor \geq k_1 +k_2.$$
By Lemma \ref{orbits1}, we now have
$$\Csum(\chi_1\chi_2,G) \leq [\tau^k,\tau^k] \leq 8q^{k^2} \leq 8q^{2L^2} \leq q^{5L^2}.$$ 
Next assume that $\eps = -$. Then the proof of Proposition \ref{tensor} shows that, for 
each $i = 1,2$, there exist a linear character $\al_i$ of $G$  and some $k_i \leq \lfloor 1.4L_2 \rfloor$
such that $\chi_i\al_i$ is a constituent of $\zeta^{k_i}$. Hence $\chi_1\chi_2\al_1\al_2$ is a constituent of $\zeta^k$ with 
$$k := k_1 + k_2 \leq \lfloor 1.4L \rfloor.$$
By Lemma \ref{orbits1}, we now have
$$\Csum(\chi_1\chi_2,G) \leq [\zeta^k,\zeta^k] \leq 2q^{k^2} \leq 2q^{2L^2} \leq q^{5L^2}.$$ 
On the other hand, if $L \geq n/5$, then in both of the cases of $\eps = +$ and $\eps = -$ we also have 
$$\Csum(\chi_1\chi_2,G) \leq (\chi_1\chi_2)(1) \leq q^{nL} \leq q^{5L^2}$$
as well.

\smallskip
(ii) Write $\chi_1 = \sum_i a_i\al_i$ and $\chi_2 = \sum_j b_j \beta_j$ with 
$\al_i,\beta_j \in \Irr(G)$ being pairwise distinct. Then $\al_i(1) \leq q^{nL_1}$ and $\beta_j(1) \leq q^{nL_2}$, and so 
$$\Csum(\al_i\beta_j,G) \leq q^{5L^2}$$
by (i). It follows that 
$$\Csum(\chi_1\chi_2) = \sum_{i,j}a_ib_j \Csum(\al_i\beta_j,G) \leq \bigl(\sum_ia_i\bigl)\bigl(\sum_j b_j\bigl)q^{5L^2} = \Csum(\chi_1,G)\Csum(\chi_2,G)q^{5L^2}.$$
\end{proof}

\begin{corol}\label{tensor3}
Let $G = \GL^\eps_n(q)$ with $n \geq 7$, $\eps = -$, and let $\chi$ be a complex character of degree at most $q^{nL}$ 
with $L \geq 1/2$. Then for any integer $m \geq 1$ we have 
$$\Csum(\chi^m,G) \leq q^{15m^2L^2}\Csum(\chi,G)^m.$$
\end{corol}

\begin{proof}
First we prove by induction on $k \geq 0$ that 
\begin{equation}\label{power-1}
  \Csum(\chi^{2^k},G) \leq q^{5(2^{2k+1}-2^{k+1})L^2}\Csum(\chi,G)^{2^k}.
\end{equation}  
The induction base $k=0$ is trivial. Now, for any $k \geq 0$ we have by Corollary \ref{tensor2}(ii) and the induction hypothesis that
$$\Csum(\chi^{2^{k+1}},G) \leq q^{5\cdot 2^{2k+2}L^2}\Csum(\chi^{2^k},G)^2 \leq q^{5(2^{2k+3}-2^{k+2})L^2}\Csum(\chi,G)^{2^{k+1}},$$
completing the induction step. 

Now we prove the desired statement by induction on $m \geq 1$, again with the trivial induction base $m=1$. For $m \geq 2$, 
take $a := 2^k$ with $k := \lfloor \log_2 m \rfloor$. According to \eqref{power-1} we have that
$$\Csum (\chi^a,G) \leq q^{10a^2L^2}\Csum(\chi,G)^a.$$
In particular, we are done if $a=m$.  Suppose that $b:= m-a \geq 1$ and apply the induction hypothesis to $b$. Then by 
Corollary \ref{tensor2}(ii) we have 
$$\Csum(\chi^m,G) \leq q^{5m^2L^2}\Csum(\chi^a,G)\Csum(\chi^b,G) \leq q^{5(m^2+2a^2+3b^2)L^2}\Csum(\chi,G)^m 
    \leq q^{15m^2L^2}\Csum(\chi,G)^m,$$
since $a > b$.
\end{proof}

\begin{propo}\label{cent}
Let $V = \FF_q^n$ be a non-degenerate symplectic or orthogonal space and let 
$G = \Sp(V)$ or $\SO(V)$, respectively. Suppose that $|\CB_V(g)| = q^k$ 
for some $g \in G$. Then 
$$|\CB_G(g)| \geq q^{(k^2-3k)/2}.$$ 
\end{propo}

\begin{proof}
First we make the following observation that 
\begin{equation}\label{order}
  |\Sp_{2k}(q)| > q^{k(2k+1)}/2,~~|\SO^\pm_m(q)| \geq q^{m(m-1)/2}/2
\end{equation}
for $k \geq 1$ and $m \geq 2$. Indeed, 
$$|\Sp_{2k}(q)| = q^{k^2}\prod^k_{i=1}(q^{2i}-1),$$
so the $\Sp$-case follows from \cite[Lemma 6.1(i)]{GLT}. The same argument applies to the $\SO$-case with $m=2k+1$ since 
$$|\SO_{2k+1}(q)| = q^{k^2}\prod^k_{i=1}(q^{2i}-1).$$
Consider the $\SO$-case with $m = 2k$. If $k \leq 3$ then \eqref{order} can be checked directly. If $k \geq 4$, 
then by \cite[Lemma 6.1(i)]{GLT} we have  
$$|\SO^\pm_{2k}(q)| \geq q^{k(k-1)}\prod^{k-1}_{i=1}(q^{2i}-1)(q^k-1) \geq q^{k(2k-1)} \cdot \frac{9}{16} \cdot \frac{15}{16} > 
    q^{k(2k-1)}/2.$$

Let $g=su$ denote the Jordan decomposition of $g$, with $s$ semisimple and $u$ unipotent. Then 
we can decompose $V = A \oplus A^\perp$, where $A = \Ker(s-1_V)$ is non-degenerate. Next, if $J_i$ denotes
the Jordan block of size $i \times i$ with eigenvalue $1$, then we have
$$u|_A = \oplus_i J_i^{r_i}$$
for some $r_i \in \ZZ_{\geq 0}$ and 
\begin{equation}\label{sum1}
  \sum_i r_i = k. 
\end{equation}  
We will view $H := \Sp(A)$, respectively
$H:=\SO(A)$,  as $H = \HC^F$, where  
$\HC := \Sp(A \otimes_{\FF_q}\overline\FF_q)$, respectively $\HC := \SO(A \otimes_{\FF_q}\overline\FF_q)$ and 
$F:\HC\to\HC$ a Frobenius endomorphism. The structure of the connected component $\CL:=\CB_\HC(u)^\circ$ is described in Theorems 3.1, 6.6, and Lemma 6.2 of \cite{LSe}; in particular,
$$\dim \CL \geq \left\{\begin{array}{ll}\sum_i ir_i^2/2 + \sum_{i < j}ir_ir_j + \sum_{2 \nmid i}r_i/2-\sum_{2|i}r_i, & \HC = \Sp,\\
    \sum_i ir_i^2/2 + \sum_{i < j}ir_ir_j - \sum_{2 \nmid i}r_i/2-\sum_{2|i}r_i, & \HC = \SO, \end{array} \right.$$
(where the additional term $\sum_{2|i}r_i$ appears only when $2|q$). Of course, if $r_i > 0$ for some $2|i$, then 
$$ir_i^2 \geq r_i^2 + r_i,$$
and so \eqref{sum1} implies that
$$\sum_i ir_i^2/2 + \sum_{i < j}ir_ir_j \geq \sum_ir_i^2/2 + \sum_{i < j}r_ir_j + \sum_{2|i}r_i/2 \geq k^2/2 + \sum_{2|i}r_i/2.$$ 
It follows that
$$\dim \CL \geq k^2/2-\sum_i r_i/2 \geq (k^2-k)/2.$$
The structure of $\CL^F$ is described in Theorems 7.1 and 7.3 of \cite{LSe}. Together with \eqref{order}, this implies that
$$|\CB_H(u)| \geq q^{\dim \CL - \sum_{i: r_i > 0}1} \geq q^{\dim \CL-k}.$$ 
Consequently,
$$|\CB_G(g)| \geq |\CB_H(u)| \geq q^{(k^2-3k)/2}.$$
\end{proof}

If $V = \FF_q^n$ is endowed with a quadratic form, then we will call a subspace $W$ of $V$ {\it non-degenerate} if 
it is non-degenerate with respect to the associated {\it bilinear} form.
 
\begin{propo}\label{restr-sp}
Let $n \geq 2$ and let $V = \FF_q^{2n}$ be endowed with a non-degenerate, symplectic or quadratic form. Accordingly,
we consider $G = \Sp(V) \cong \Sp_{2n}(q)$ in the symplectic case, and $G = \SO(V) \cong \SO^\al_{2n}(q)$
or $G = \Om(V) = \Omega^\al_{2n}(q)$ in the quadratic case. Let $H \cong \Sp_{2n-2}(q)$, respectively 
$\SO^\beta_{2n-2}(q)$ or $\Om^\beta_{2n-2}(q)$,  be the subgroup of  $G$ consisting of all elements that act trivially on a non-degenerate $2$-dimensional subspace $W$ of $V$, where $\alpha,\beta = \pm$ are fixed. 
Let $\chi \in \Irr(G)$ be of degree at most $q^D$ for some $D \geq 1$. Then
$$[\chi|_H,\chi|_H]_H \leq q^{2+\sqrt{41n+16D}}.$$ 
\end{propo}

\begin{proof}
(i) For any element $g \in G$, let $g^G$ denote the $G$-conjugacy class of $g$ in $G$, and let $N(g)$ denote the number of 
non-degenerate $2$-dimensional subspaces of $V$ (of fixed type $\al\beta$ if $V$ is quadratic) 
on which $g$ acts trivially. Here we prove that
\begin{equation}\label{int-sp}
  \frac{|g^G \cap H|}{|g^G|} \cdot \frac{|G|}{|H|} \leq |I(W)| \cdot N(g),
\end{equation}
where $I(W) \cong \Sp_2(q)$ if $V$ is symplectic, and $I(W) \cong \GO^{\al\beta}_2(q)$ if 
$V$ is quadratic. Indeed, let 
$$\XC = \XC_g := \{ (h,x) \in H \times G \mid h=x^{-1}gx \}.$$
First we consider the projection
$$\pi_1: \XC \to g^G \cap H,~~(h,x) \mapsto h$$
which is surjective, with fibers of size $|\CB_G(g)| = |G|/|g^G|$. It follows that
$$|\XC| = |G| \cdot \frac{|g^G \cap H|}{|g^G|}.$$
Next, consider the map
$$\pi_2: (h,x) \mapsto x(W).$$
Note that $g = xhx^{-1}$ acts trivially on $x(W)$ for any $(h,x) \in \XC$. The fiber $\pi_{2}^{-1}(x(W))$ consists of 
pairs $(h,y) = (y^{-1}gy,y) \in \XC$ with $x^{-1}y(W) = W$, and so of size at most $|H| \cdot |I(W)|$. Thus
$$|\XC| \leq N(g) \cdot |H| \cdot |I(W)|,$$
and the claim follows.  

\smallskip
(ii) Next we observe that if $|\CB_V(g)| = q^k$, then 
\begin{equation}\label{ng-sp}
  |I(W)| \cdot N(g) \leq (q^k-1)q^{k-1} < q^{2k-1}.
\end{equation}  
Indeed, let $(\cdot,\cdot)$ denote the bilinear form on $V$, and let $W = \langle u,v \rangle_{\FF_q}$ be a fixed  
non-degenerate $2$-space (of type $\al\beta$ if $V$ is quadratic) on which $g$ acts trivially, where 
we choose $v$ such that $(u,v) \neq 0$. Then $|I(W)| \cdot N(g)$ is the number of linear isometries $f:W \to \CB_V(g)$,
and for each such $f$, we have at most $q^k-1$ choices for $f(u)$. Next, $(f(u),f(v)) = (u,v) \neq 0$, and so 
$$f(v) \in W \smallsetminus f(u)^\perp \subseteq \CB_V(g) \smallsetminus f(u)^\perp.$$ 
It follows that $\CB_V(g) + f(u)^\perp = V$ and so
$\dim(\CB_V(g) \cap f(u)^\perp) = k-1$. Now one can see that the number of choices for $f(v)$ is 
at most $|\CB_V(g) \cap f(u)^\perp| = q^{k-1}$.

\smallskip
(iii) We call $g \in G$ {\it good} if 
$$|g^G| \cdot |\chi(g)|^2 < q^{-n}|H|$$ 
and {\it bad} otherwise. Accordingly we can write
\begin{equation}\label{sum2}
  [\chi|_H,\chi|_H]_H = \sum_{g^G:\ g\ \mbox{\tiny good}}\frac{|g^G \cap H| \cdot |\chi(g)|^2}{|H|}
    +  \sum_{g^G:\ g\ \mbox{\tiny bad}}\frac{|g^G \cap H| \cdot |\chi(g)|^2}{|H|}.
\end{equation}    
By Theorems 3.12, 3.13, 3.16, 3.18, 3.21, and 3.22 of \cite{FG}, $|\Irr(G)| \leq 15.2q^n$. 
It follows that the sum over good classes in \eqref{sum2} 
is at most $15.2$. If $g$ is bad, then we have 
$$|\CB_G(g)|  \leq q^n|\chi(g)|^2[G:H] < q^{5n-1+2D}.$$
It follows from Proposition \ref{cent} that 
$$k \leq \frac{3}{2} + \sqrt{10n+4D+\frac{1}{4}}$$
if $|\CB_V(g)| = q^k$. Also, \eqref{int-sp} and \eqref{ng-sp} imply that
$$\frac{|g^G \cap H|}{|H|} \leq \frac{|g^G| \cdot |I(W)| \cdot N(g)}{|G|} < q^{2k-1} \cdot \frac{|g^G|}{|G|}.$$
Also note that 
$$\sum_{g^G} \frac{|g^G|}{|G|} \cdot |\chi(g)|^2 = \frac{1}{|G|}\sum_{g \in G}|\chi(g)|^2 = 1.$$
Applying the above estimates to the sum over bad classes in \eqref{sum2}, we obtain that
$$[\chi|_H,\chi|_H]_H \leq 15.2 + q^{2+\sqrt{40n+16D+1}} < q^{2+\sqrt{41n+16D}},$$
as stated.
\end{proof}

We will also need the following variant of Proposition \ref{restr-sp}:

\begin{propo}\label{restr-so}
Let $q$ be an odd prime power, $V = \FF_q^{2n+1}$ be a non-degenerate quadratic space with $n \geq 3$, and let 
$G = \SO(V) \cong \SO_{2n+1}(q)$ or $\Om(V) \cong \Om_{2n+1}(q)$. Let $H \cong \SO^+_{2n}(q)$, 
respectively $H \cong \Om^+_{2n}(q)$, be the subgroup of 
$G$ consisting of all elements that act trivially on a non-degenerate $1$-dimensional subspace $W$ of $V$. Let 
$\chi \in \Irr(G)$ be of degree at most $q^D$ for some $D \geq 1$. Then
$$[\chi|_H,\chi|_H]_H \leq q^{2+\sqrt{7n+5D}}.$$ 
\end{propo}

\begin{proof}
(i) For any element $g \in G$, let $N(g)$ denote the number of 
non-degenerate $1$-dimensional subspaces $W$ of $V$ on which $g$ acts trivially and such that $W^\perp$ is of type $+$. 
First we prove that
\begin{equation}\label{int-so}
  \frac{|g^G \cap H|}{|g^G|} \cdot \frac{|G|}{|H|} \leq 2N(g).
\end{equation}
Indeed, let 
$$\XC := \{ (h,x) \in H \times G \mid h=x^{-1}gx \}.$$
As in the proof of Proposition \ref{restr-so}, by considering the projection
$$\pi_1: \XC \to g^G \cap H,~~(h,x) \mapsto h$$
we see that 
$$|\XC| = |G| \cdot \frac{|g^G \cap H|}{|g^G|}.$$
Next, consider the map
$$\pi_2: (h,x) \mapsto x(W).$$
Note that $g = xhx^{-1}$ acts trivially on $x(W)$ for any $(h,x) \in \XC$. The fiber $\pi_{2}^{-1}(x(W))$ consists of 
pairs $(h,y) = (y^{-1}gy,y) \in \XC$ with $x^{-1}y(W) = W$, and so of size at most $2|H|$. Thus
$$|\XC| \leq N(g) \cdot 2|H|,$$
and the claim follows.  

We also observe that if $|\CB_V(g)| = q^k$, then $N(g) < q^k$. 

\smallskip
(ii) Call $g \in G$ {\it good} if 
$$|g^G| \cdot |\chi(g)|^2 < q^{-n}|H|$$ 
and {\it bad} otherwise. Accordingly we can write
\begin{equation}\label{sum3}
  [\chi|_H,\chi|_H]_H = \sum_{g^G:\ g\ \mbox{\tiny good}}\frac{|g^G \cap H| \cdot |\chi(g)|^2}{|H|}
    +  \sum_{g^G:\ g\ \mbox{\tiny bad}}\frac{|g^G \cap H| \cdot |\chi(g)|^2}{|H|}.
\end{equation}    
By Theorems 3.17 and 3.19 of \cite{FG}, $|\Irr(G)| \leq 7.3q^n$. It follows that the sum over good classes in \eqref{sum2} 
is at most $7.3$. If $g$ is bad, then we have 
$$|\CB_G(g)|  \leq q^m|\chi(g)|^2[G:H] < q^{3n+2D+0.06}$$
(where we have used the estimate $(q^n+1) < 1.04q^n < q^{n+0.06}$).
It follows from Proposition \ref{cent} that 
$$k \leq \frac{3}{2} + \sqrt{6n+4D+2.37}$$
if $|\CB_V(g)| = q^k$. Also, \eqref{int-sp} and \eqref{ng-sp} imply that
$$\frac{|g^G \cap H|}{|H|} \leq \frac{|g^G| \cdot 2N(g)}{|G|} < q^{k+0.64} \cdot \frac{|g^G|}{|G|}.$$
Also note that 
$$\sum_{g^G} \frac{|g^G|}{|G|} \cdot |\chi(g)|^2 = \frac{1}{|G|}\sum_{g \in G}|\chi(g)|^2 = 1.$$
Applying the above estimates to the sum over bad classes in \eqref{sum2}, we obtain that
$$[\chi|_H,\chi|_H]_H \leq 7.3 + q^{2.14+\sqrt{6n+4D+2.37}} < q^{2+\sqrt{7n+5D}},$$
as stated.
\end{proof}

Next we endow $V = \FF_q^{2n}$ with a non-degenerate symplectic, respectively quadratic form with a Witt basis 
$(e_1, \ldots,e_n,f_1, \ldots ,f_n)$, where we assume furthermore that $\langle e_1, \ldots,e_n \rangle_{\FF_q}$ is 
totally singular. Then let $G_n := \Sp(V) \cong \Sp_{2n}(q)$, respectively $\SO(V) \cong \SO^+_{2n}(q)$.
Consider the parabolic subgroup
$$P_n := \Stab_{G_n}(\langle e_1, \ldots,e_n \rangle_{\FF_q}) = Q_nL_n$$
of $G_n$, with abelian unipotent radical $Q_n$ and  Levi subgroup $L_n \cong \GL_n(q)$. If $2|q$, then $P_n$ is 
contained in $\Om(V)$, and, abusing the notation, we will also consider the case 
$G_n := \Om(V) = \Om^+_{2n}(q)$. 

For any $1 \leq m \leq n$, we will also consider the {\it standard} subgroup
$$G_m = \Stab_{G_n}(e_{m+1}, \ldots,e_n,f_{m+1}, \ldots ,f_n)$$
(which is isomorphic to $\Sp_{2m}(q)$, respectively $\SO^+_{2m}(q)$ if $G_n = \SO(V)$, and 
$\Om^+_{2m}(q)$ if $G_n = \Om(V)$ and $2|q$), and its parabolic subgroup $P_m$ and Levi subgroup
$L_m \cong \GL_m(q)$. 

\begin{propo}\label{parab}
Let $n \geq 7$ and $0 \leq L \leq n/5$. Let $\varphi \in \Irr(P_n)$ be of degree at most $q^{nL}$. If 
$1 \leq m \leq n-\lfloor 1.4L \rfloor$, then the restriction of $\varphi$ to $P_m$ contains a linear character of $P_m$.
\end{propo}

\begin{proof}
(i) Let $\SC_n(q)$ denote the set of symmetric $n \times n$-matrices over $\FQ$, and 
let $\AC_n(q)$ denote the set of anti-symmetric $n \times n$-matrices over $\FQ$ with zero diagonal. 
Fix a primitive $p^{\mathrm {th}}$ root $\varep$ of $1$ in $\mathbb C$.
First we consider the case $G_n = \Sp_{2n}(q)$. Then 
$$Q_n = \left\{ [I,X] := \begin{pmatrix}I_n & X\\0 & I_n\end{pmatrix} 
    \mid X \in \SC_n(\FQ) \right\}.$$
If $2\nmid q$, then any linear character of $Q_n$ is of the form 
$$\lam_B~:~[I,X] \mapsto \varep^{\tr_{\FQ/\FF_p}(\Tr(BX))}$$
for some $B \in \SC_n(q)$. Any such $B$ defines a quadratic form of rank $j=j(B)$ and type $\eps = \pm$, and
\begin{equation}\label{stab1}
 J_B := \Stab_{L_n}(\lam_B) = [q^{j(n-j)}] \rtimes (\GO^\eps_j(q) \times \GL_{n-j}(q)),
\end{equation}  
where $[q^a]$ denotes an elementary abelian group of order $q^a$. If $q$ is even, then any linear character of $Q_n$ is of the form 
$$\lam_B~:~[I,X] \mapsto \eps^{\tr_{\FQ/\FF_p}(\Tr(BX))}$$
for some $B \in M_n(q)$, and $\lam_B = \lam_{B'}$ if and only if $B-B' \in \AC_n(q)$. Now, 
for each $B = (b_{ij}) \in M_n(q)$ we define the quadratic form $q_B$ on the 
space $\FF_q^n = \langle f_1, \ldots ,f_n\rangle_{\FQ}$ such that
$q_B(f_{i}) = b_{ii}$ and the associated bilinear form has $B + \tw t B$ as Gram 
matrix in the given basis. Then \eqref{stab1} again holds with $r$ being the rank of $q_B$,
see e.g. \cite[Lemma 3.1]{GT2}.

Suppose now that $G_n = \SO^+_{2n}(q)$, or $2|q$ and $G_n = \Om^+_{2n}(q)$. Then 
$$Q_n  = \left\{ [I,X] := \begin{pmatrix}I_n & X\\0 & I_n\end{pmatrix} 
  \mid X \in \AC_n(\FQ) \right\}.$$
If $q$ is odd, then any linear character of $Q_n$ is of the form 
$$\lam_B~:~[I,X] \mapsto \eps^{\tr_{\FQ/\FF_p}(\Tr(BX))}$$
for some $B \in \AC_n(q)$. If $q$ is even, then any linear character of $Q_n$ is of the form 
$$\lam_B~:~[I,X] \mapsto \eps^{\tr_{\FQ/\FF_p}(\Tr(BX))}$$
for some $B \in M_n(q)$, and $\lam_B = \lam_{B'}$ if and only if $B-B' \in \SC_n(q)$. Let the even integer $j=j(B)$ denote the rank
of $B$ if $2 \nmid q$, and of $B - \tw tB$ if $2|q$. Then again we have 
\begin{equation}\label{stab2}
  J_B := \Stab_{L_n}(\lam_B) = [q^{j(n-j)}] \rtimes (\Sp_j(q) \times \GL_{n-j}(q)).
\end{equation}  

\smallskip
(ii) Let $\OC$ denote the $P_n$-orbit of $\lam_B$ with $j = j(B)$. Here we show that 
\begin{equation}\label{length1}
  |\OC| \geq \left\{ \begin{array}{ll}\max(q^{j(n-j)},q^{j(n-(j-1)/2)-3}),& G_n = \Sp(V),\\
  \max(q^{j(n-j)},q^{j(n-(j+1)/2)-2}),& G_n = \SO(V) \mbox{ or }\Om(V). \end{array} \right.  
\end{equation} 
The statement is obvious for $0 \leq j \leq 2$, so we will assume $j \geq 3$. If $G_n = \Sp(V)$, then the bound 
$|\OC| \geq q^{j(n-(j-1)/2)-3} > q^{j(n-j)}$ follows from \eqref{stab1} and the estimates
$$|\GL_n(q)| \geq (9/32)q^{n^2},~~|\GO^\pm_j(q)| < 2q^{j(j-1)/2},~~|\GL_{n-j}(q)| \leq q^{(n-j)^2}$$
(with the first one following \cite[Lemma 6.1(i)]{GLT}). If $G_n = \SO(V)$ or $\Om(V)$, then the bound 
$|\OC| \geq q^{j(n-(j+1)/2)-2} > q^{j(n-j)}$ then follows from \eqref{stab1} and the estimates
$$|\GL_n(q)| \geq \frac{9}{16} \cdot \frac{q-1}{q} \cdot q^{n^2},
    ~~|\Sp_j(q)| \leq \frac{q^2-1}{q^2} \cdot q^{j(j+1)/2},~~|\GL_{n-j}(q)| \leq q^{(n-j)^2}$$
(with the first one following \cite[Lemma 6.1(i)]{GLT}). 

\smallskip
(iii) Assume now that $\lam = \lam_B$ occurs in $\varphi|_{Q_n}$.  By Clifford's theorem, 
$\varphi=\Ind^{P_n}_T(\hat\lam)$ for some irreducible character $\hat\lam$ of 
$$T := \Stab_{P_n}(\lam) = Q_n \rtimes J_B$$ 
that lies above $\lam$. Let $\mu$ be an irreducible constituent of the restriction of 
$\hat\lam$ to the subgroup $\GL_{n-j}(q)$ of $J_B$.
Then we must have that 
\begin{equation}\label{length2}
  |\OC| \leq |\OC| \cdot \mu(1) \leq \varphi(1) \leq q^{nL} \leq q^{n^2/5}.
\end{equation}  
First note that if $n/2 \leq j \leq n$, then 
$$\min(j(n-(j-1)/2)-3,j(n-(j+1)/2)-2) > n^2/5.$$
Hence \eqref{length1} and \eqref{length2} imply that $j \leq n/2$. This, together with \eqref{root1} and
\eqref{length2} actually shows that
\begin{equation}\label{j1}
  j \leq 1.4L.
\end{equation}   
In particular, $n-j \geq 3.6n/5 > 4$.

Writing $\cl(\mu) = i$, we aim to show that 
\begin{equation}\label{j2}
  i+j \leq 1.4L.
\end{equation}
Suppose first that $i \geq (n-j)/2$. Then $\mu(1) \geq q^{(n-j)^2/4-2}$ by \cite[Theorem 1.1]{GLT}, and so
by \eqref{length1} we have
$$\varphi(1) \geq q^{j(n-j)+(n-j)^2/4-2}.$$
Note that $h(t) := t(1-t)+(1-t)^2/4-1/5-2/49 > 0$ on the interval $[0,1/2]$. Since $0 \leq j/n < 1/2$ by \eqref{j1}, we 
see that 
$$j(n-j)+(n-j)^2/4-2 = n^2 \cdot h(j/n) + n^2/5 + 2n^2/49-2 > n^2/5 + 2n^2/49-2 \geq n^2/5,$$
violating \eqref{length2}. We have shown that $i < (n-j)/2$, and so
$$i+j < n/2 +j/2 \leq n/2 +0.7L$$
by \eqref{j1}. This, together with \eqref{root1} and \eqref{length2}, yields \eqref{j2}. 

\smallskip
(iv) Without loss we may assume that $B = \begin{pmatrix} B_1 & 0\\0 & 0 \end{pmatrix}$ for a suitable $j \times j$-matrix $B_1$.
Then we can choose $G_{n-j}$ to be the subgroup
$$\Stab_{G_n}(e_1, \ldots ,e_j,f_1, \ldots,f_j)$$
and the subgroup $\GL_{n-j}(q)$ in $J_B$ to be a Levi subgroup $L_{n-j}$ of the parabolic subgroup
$$P_{n-j} = \Stab_{G_{n-j}}(\langle e_{j+1}, \ldots,e_n \rangle_{\FF_q})$$
of $G_{n-j}$. As in the proof of Proposition \ref{tensor}, the condition $\cl(\mu) = i$ implies that the restriction of $\mu$ to 
a standard subgroup $\GL_{n-i-j}(q)$ of $\GL_{n-j}(q)$ contains a linear character $\nu$ of $\GL_{n-i-j}(q)$. 
As above, we can choose $G_{n-i-j}$ to be the subgroup
$$\Stab_{G_n}(e_1, \ldots ,e_{i+j},f_1, \ldots,f_{i+j})$$
and the subgroup $\GL_{n-i-j}(q)$ to be a Levi subgroup $L_{n-i-j}$ of the parabolic subgroup
$$P_{n-i-j} = \Stab_{G_{n-i-j}}(\langle e_{i+j+1}, \ldots,e_n \rangle_{\FF_q})$$
of $G_{n-i-j}$. Note that the unipotent radical $Q_{n-i-j}$ of $P_{n-i-j}$ is contained in $Q_n$ and consists of
matrices $[I,Y]$, where $Y = \begin{pmatrix} 0 & 0\\0 & Y_1 \end{pmatrix}$ for a suitable $(n-i-j) \times (n-i-j)$-matrix $Y_1$
over $\FF_q$, and so $Q_{n-i-j} \leq \Ker(\lam)$. 

Now $\varphi|_{P_{n-i-j}}$ contains $\hat\lam|_{P_{n-i-j}}$, and the latter has been shown to contain a linear character,
trivial at $Q_{n-i-j}$ and equal to $\nu$ at $\GL_{n-i-j}(q)$. As $i+j \leq \lfloor 1.4L \rfloor$ by \eqref{j2}, we are done.
\end{proof}

Now we can prove the main result of this section:

\begin{theor}\label{restr-gl}
There is an explicit absolute constant $A >0$ (which can be taken to be $705$) such that the following statement holds.
Let $q$ be a prime power, $n \geq 1$, $G_n = \Sp_{2n}(q)$ or $\SO^+_{2n}(q)$, or $\Omega^+_{2n}(q)$ with $2|q$, 
and let $L_n = \GL_n(q)$ be a Levi subgroup of the parabolic subgroup
$P_n$ of $G_n$. Suppose that $\chi$ is a complex character of $G_n$ of degree at most $q^{nL}$, where $L \geq 1$ and 
$$n \geq \max(7.6L, 7+1.4L).$$
Then 
$$\Csum(\chi,L_n) \leq q^{\sqrt{AnL^3}}\Csum(\chi,G_n).$$
Furthermore, for any $m \geq 1$ we have 
$$\Csum(\chi^m,L_n) \leq q^{m\sqrt{AnL^3}+15m^2L^2}\Csum(\chi,G_n)^m.$$  
\end{theor}

\begin{proof}
(i) By Proposition \ref{restr-sp}, there is an absolute constant $A_1 >0$ (which can be taken to be $69$) such that 
$$\Csum(\psi,G_{m-1}) \leq [\psi|_{G_{m-1}},\psi|_{G_{m-1}}]_{G_{m-1}} \leq q^{\sqrt{A_1nL}}$$
for all $7 \leq n$, $7 \leq m \leq n$, and for $\psi \in \Irr(G_m)$ of degree at most $q^{nL}$.  Setting 
$$k = \lfloor 1.4L \rfloor,~~A_2 = (1.4)^2A_1 (\approx 135.3),$$
and applying Lemma \ref{mults}(i), we see that
\begin{equation}\label{restr-g}
  \Csum(\chi,G_{n-k}) \leq q^{\sqrt{A_2nL^3}}\Csum(\chi,G_n).
\end{equation}
Next we show that if $n \geq \max(7,5L)$ and $\chi(1) \leq q^{nL}$ then
\begin{equation}\label{restr-p1}
  \Csum(\chi,P_n) \leq q^{\sqrt{A_3nL^3}}\Csum(\chi,G_n).
\end{equation}
For, if $\varphi \in \Irr(P_n)$ is any irreducible constituent of $\chi|_{P_n}$, then 
$\varphi|_{P_{n-k}}$ contains a linear constituent $\lam \in \Irr(P_{n-k})$ by Proposition \ref{parab}. It follows 
that 
$$\Csum(\chi,P_n) \leq \Clin(\chi,P_{n-k}).$$
On the other hand, by Lemma \ref{coset}, the multiplicity of each linear $\lam \in \Irr(P_{n-k})$ is at most
$$\sqrt{|P_{n-k} \backslash G_{n-k}/P_{n-k}|} \leq \sqrt{n}$$
in the restriction to $P_{n-k}$ of any irreducible character of $G_{n-k}$. Hence,
$$\Clin(\chi,P_{n-k}) \leq \sqrt{n}\cdot \Csum(\chi,G_{n-k}) \leq q^{\sqrt{A_3nL^3}}\Csum(\chi,G_n)$$
by \eqref{restr-g}, with $A_3 = 148$, establishing \eqref{restr-p1}.

\smallskip
(ii) Recall we have proved \eqref{restr-p1} for any (not necessarily irreducible) character of $G_m$ of degree at most
$q^{mL}$, provided that $m \geq \max(7,5L)$. Consider any irreducible constituent $\al \in \Irr(G_{n-k})$ of 
$\chi|_{G_{n-k}}$. Then 
$$\al(1) \leq q^{nL} = q^{(n-k)M},$$
where 
$$M = \frac{nL}{n-k} \leq \frac{L}{1- 1.4L/n} < 1.226L,$$
since $n \geq 7.6L$. Now we have 
$$n-k > \max(7,5M)$$
and so \eqref{restr-p1} can be applied to $\al$ to yield 
$$\Csum(\al,P_{n-k}) \leq q^{\sqrt{A_3(n-k)M^3}} \leq q^{\sqrt{A_4nL^3}},$$
with $A_4 \leq 1.504A_3$ can be taken to be $222.6$. Using this together with \eqref{restr-g} and Lemma \ref{mults}, we get 
\begin{equation}\label{restr-p2}
  \Csum(\chi,P_{n-k}) \leq q^{\sqrt{A_4nL^3}+\sqrt{A_2nL^3}}\Csum(\chi,G_n) \leq q^{\sqrt{A_5nL^3}}\Csum(\chi,G_n),
\end{equation}  
with $A_5 = 705$.

\smallskip
(iii) By Lemma \ref{abelian}, the multiplicity of each linear $\mu \in \Irr(L_{n-k})$ is at most $1$
in the restriction to $L_{n-k}$ of any irreducible characters of $P_{n-k} = Q_{n-k} \rtimes L_{n-k}$. Hence,
$$\Clin(\chi,L_{n-k}) \leq \Csum(\chi,P_{n-k}) \leq q^{\sqrt{A_5nL^3}}\Csum(\chi,G_n)$$
by \eqref{restr-p2}. On the other hand, if $\gam \in \Irr(L_n)$ is an irreducible constituent of $\chi|_{L_n}$, then 
$\gam(1) \leq \chi(1) \leq q^{nL}$, and so by Proposition \ref{tensor}(i), its restriction to $L_{n-k}$ contains a 
linear character of $L_{n-k}$. It follows that
$$\Csum(\chi,L_n) \leq \Clin(\chi,L_{n-k}),$$
and so we obtain the first statement of the theorem by taking $A = A_5 = 705$. The second statement then follows from 
Corollary \ref{tensor3}.
\end{proof}

For the next statement, we note that $\SO(V) = \SO_{2n+1}(q)$ with $2 \nmid q$ contains a standard subgroup $H \cong \SO^+_{2n}(q)$
which fixes a non-singular vector in $V$. Likewise, $\SO(V) = \SO^-_{2n+2}(q)$ contains a standard subgroup $H \cong \SO^+_{2n}(q)$
which acts trivially on a non-degenerate $2$-dimensional subspace of $V$.
Furthermore, $\Om(V) = \Om^-_{2n+2}(q)$ contains a subgroup
$H \cong \SO^+_{2n}(q)$, which fixes an orthogonal decomposition $V = V_1 \oplus V_2$, where $V_2$ is a non-degenerate 
$2$-dimensional subspace of type $-$ and has $[H,H] \cong \Om^+_{2n}(q)$ acting trivially on $V_2$; we will 
refer to any such subgroup $H$ as a standard $\SO^+_{2n}(q)$-subgroup of $\Om^-_{2n+2}(q)$. 

\begin{corol}\label{restr-gl2}
There is an explicit absolute constant $B >0$ (which can be taken to be $1216$) such that the following statement holds.
Let $q$ be a prime power, $n \geq 1$, $G := \SO_{2n+1}(q)$ with $2 \nmid q$, or $G := \SO^-_{2n+2}(q)$, or
$G:=\Om^-_{2n+2}(q)$. Let 
$L_n = \GL_n(q)$ be a Levi subgroup of the parabolic subgroup
$P_n$ of a standard subgroup $H \cong \SO^+_{2n}(q) < G$.
Suppose that $\chi$ is a complex character of $G$ is of degree at most $q^{nL}$, 
where $L \geq 1$ and 
$$n \geq \max(7.6L, 7+1.4L).$$
Then 
$$\Csum(\chi,L_n) \leq q^{\sqrt{BnL^3}}\Csum(\chi,G).$$
Furthermore, for any $m \geq 1$ we have 
$$\Csum(\chi^m,L_n) \leq q^{m\sqrt{BnL^3}+15m^2L^2}\Csum(\chi,G)^m.$$  
\end{corol}

\begin{proof}
Set $\tilde H := \Om^+_{2n}(q)$ if $G = \Om^-_{2n+2}(q)$, and $\tilde H := H$ otherwise.
As in the proof of Theorem \ref{restr-gl}, Propositions \ref{restr-sp} and \ref{restr-so} applied to the pair
$G > \tilde H$ imply that
$$\Csum(\chi,H) \leq \Csum(\chi,\tilde H) \leq q^{\sqrt{A_1nL}}\Csum(\chi,G),$$
where $A_1$ can be taken to be $69$. On the other hand, by Theorem \ref{restr-gl},
$$\Csum(\al,\GL_n(q)) \leq q^{\sqrt{AnL^3}}$$
for any irreducible constituent $\al$ of $\chi|_H$, where $A$ can be taken to be $705$. Hence the statement follows from
Lemma \ref{mults}.
\end{proof}

We will also note the following consequence of Theorem \ref{restr-gl} and Corollary \ref{restr-gl2}:

\begin{corol}\label{restr-gl3}
There is an explicit absolute constant $C >0$ (which can be taken to be $1696$) such that the following statement holds.
Let $q$ be any odd prime power, $n \geq 1$, $\tilde G := \SO_{2n+1}(q)$ or $\SO^+_{2n}(q)$, and 
let $G := [\tilde G,\tilde G] \cong \Om_{2n+1}(q)$, respectively $\Om^+_{2n}(q)$. Let 
$L_n = \GL_n(q)$ be a Levi subgroup of the parabolic subgroup
$P_n$ of $\tilde G$.
Suppose that $\chi \in \Irr(G)$ is of degree at most $q^{nL}$, 
where $L \geq 1$ and 
$$n \geq \max(8.5L, 7+1.6L).$$
Then 
$$\Csum(\chi,L_n \cap G) \leq q^{\sqrt{CnL^3}}.$$
\end{corol}

\begin{proof}
Let $\tilde \chi \in \Irr(\tilde G)$ be lying above $\chi$, so that 
$$\tilde \chi(1) \leq 2\chi(1) \leq  q^{nL+1} \leq q^{nL_1},$$
where $L_1 = (10/9)L$. Applying Theorem \ref{restr-gl} and Corollary  \ref{restr-gl2} and \ref{restr-so} 
to $\tilde \chi$, we see that 
$$\Csum(\chi,L_n \cap G) \leq \Csum(\tilde \chi,L_n \cap G) \leq 2\Csum(\tilde\chi,L_n) \leq 2q^{\sqrt{Bn(L_1)^3}} \leq 
    q^{\sqrt{CnL^3}},$$
where $B = 1216$ and $C$ can be taken to be $1696$. 
\end{proof}

\section{Ineffective bounds on character values}

\newcommand{\cT}{\mathcal{T}}
A \emph{triple} $(G,q,n)$ consists of a finite group and two positive integers.  
We say a set $\cT$ of triples is $B$-\emph{bounded} for some $B > 0$, if
for every $(G,q,n)\in \cT$,  every
$L\ge 1$,  every character $\chi$ of $G$ with $\chi(1)\le q^{nL}$, and every positive integer $m$ we have
$$\Csum(\chi^m,G) \le q^{B(mn^{1/2}L^{3/2}+m^2L^2)} \Csum(\chi,G)^m,$$
with $\Csum(\cdot,\cdot)$ as defined in \eqref{d-sum1}.
\begin{lemma}\label{modify}
Fix $C>0$.
Let $\cT_1$ and $\cT_2$ be sets of triples such that for all $(G_1,q,n)\in \cT_1$ there exists a triple
$(G_2,q,n)$ for which either $G_2$ is isomorphic to a subgroup of $G_1$ of index $\le q^C$ or $G_1$ is isomorphic to
a subgroup of $G_2$ of  index $\le q^C$. If $\cT_2$ is $B_2$-bounded, then $\cT_1$ is $B_1$-bounded
(for some $B_1$ depending on $B_2$ and $C$).
\end{lemma}

\begin{proof}
We will prove that $\cT_1$ is $B_1$-bounded for $B_1 := B_2(C+1)^2+C$.
Assume $(G_1,q,n)\in \cT_1$, $L_1\ge 1$, and 
$\chi_1$ is a character of $G_1$ with $\chi_1(1) \le q^{nL_1}$.

Suppose $G_1 \leq G_2$ for some $(G_2,q,n) \in \cT_2$ and $[G_2:G_1] \le q^C$.
Let $L_2 := (C+1)L_1$ and $\chi_2 := \Ind_{G_1}^{G_2}(\chi_1)$, so that $\chi_2(1) \le q^{nL_1+C} \le q^{nL_2}$.  
Since a tensor product of induced representations naturally contains the induced representation of the tensor product,
by Lemma \ref{mults}(ii) we have
$$\begin{aligned}
    \Csum(\chi_1^m,G_1) & \le \Csum(\chi_2^m,G_1) \le \Csum(\chi_2^m,G_2)[G_2:G_1]\\
    & \leq q^{B_2(mn^{1/2}L_2^{3/2}+m^2L_2^2)+C}\Csum(\chi_2,G_2)^m\\ 
    & \le q^{B_1(mn^{1/2}L_1^{3/2}+m^2L_1^2)}\Csum(\chi_1,G_1)^m.
    \end{aligned}$$
Suppose, on the other hand, that $G_2 \leq G_1$ for some $(G_2,q,n) \in \cT_2$ and  $[G_1:G_2] \le q^C$.
Then again using Lemma \ref{mults}(ii) we see that
$$\begin{aligned}
    \Csum(\chi_1^m,G_1) & \le \Csum(\chi_1^m,G_2)\\
    & \le q^{B_2(mn^{1/2}L_1^{3/2}+m^2L_1^2)}\Csum(\chi_1,G_2)^m\\
    & \leq q^{B_1(mn^{1/2}L_1^{3/2}+m^2L_1^2)}\Csum(\chi_1,G_1)^m,
    \end{aligned}$$
so $\cT_1$ is indeed $B_1$-bounded.
\end{proof}

Corollary~\ref{tensor3} implies that the set of triples of the form $(\GL_n^\epsilon(q),q,n)$ is $B_1$-bounded,
with $B_1=15$. Using Theorem~\ref{restr-gl} and Corollary~\ref{restr-gl2} as well,
we see that the set of triples of the form $(\Sp_{2n}(q),q,n)$, $(\SO^+_{2n}(q),q,n)$,  
$(\SO_{2n+1}(q),q,n)$, or $(\SO^-_{2n+2}(q),q,n)$, is $B_2$-bounded, with $B_2 = 1216$.  By Lemma~\ref{modify}, for a fixed 
$a > 0$, the set of triples $(G,q,n)$, with $G$ being any $(q,n,a)$-classical group, is $B_3$-bounded for some 
$B_3$ depending on $a$. It follows that all classical groups belong to $B$-bounded triples, with $B \geq1$ suitably chosen.

\begin{theor}
\label{classical-cases}
For every $\varep\in (0,1)$, there exists an explicit constant $\delta > 0$ such that the following statement holds.
If $G$ is a classical group and $g\in G$ satisfies
$|\CB_G(g)| \leq |G|^\delta$, then
$$|\chi(g)| \leq \chi(1)^\varep$$
for all $\chi \in \Irr(G)$.

\end{theor}

\begin{proof}
By the above discussion, we can fix $B \geq 1$ so that $(G,q,n)$ is $B$-bounded, if the classical group $G$ has parameter $n$.
By choosing $\delta$ small enough, we may assume $n$ is as large as we wish, since $|\CB_G(g)| \le |G|^\delta$ never occurs for bounded
$n$: the order of the centralizer of any non-central element $g \in G$ is bounded
below by $\max(2,(q-1)/2)$, whereas the order of $G$ is bounded above by $O(q^{(n+1)(2n+1)})$. 
We also assume that $|\CB_G(g)| < q^{\delta_1 n^2/2}$
where $\delta_1>0$ can be taken as small as we wish.
We choose $\delta_1$ and $\delta_2$ so that 
$$\delta_2 \leq \frac{2\varep^2}{81 B^2} \leq \frac{2\varep}{81 B},~\frac{\delta_1}{\delta_2} \leq \frac{\varep}{3}.$$
We use the centralizer bound
$$|\varphi(g)| \le |\CB_G(g)|^{1/2} < q^{\delta_1 n^2}$$
for every $\varphi\in \Irr(G)$, and let 
$$L = \max(\log_{q^n} \chi(1),1).$$ 
As the minimum dimension of any non-linear irreducible character of a classical group is  greater than $q^{n/3}$ by the Landazuri-Seitz bound
\cite{LaSe}, we have 
\begin{equation}\label{b1}
  \chi(1) > q^{Ln/3}.
\end{equation}  
We may assume 
\begin{equation}\label{b2}
  L \le \delta_2 n/2,
\end{equation}   
since otherwise, the desired character estimate follows immediately from the centralizer bound.

Fix an integer $m \geq 2$ such that
$$\frac {\delta_2 n}{L} \le m < \frac{2\delta_2 n}L.$$
This is possible because of \eqref{b2}.
Then the $B$-boundedness implies
$$|\chi(g)|^m = |\chi^m(g)| \leq \Csum(\chi^m,G)q^{\delta_1 n^2/2} \le q^{B(mn^{1/2}L^{3/2}+m^2L^2)+\delta_1n^2/2}.$$
Hence, by \eqref{b1} it suffices to prove
$$B(n^{1/2}L^{3/2}+mL^2)+\frac{\delta_1 n^2}{2m} < \frac{\varep Ln}3.$$
We do this by combining the inequalities
$$B n^{1/2} L^{3/2} = Ln (B\sqrt{L/n}) \leq LnB\sqrt{\delta_2/2} \leq \frac{\varep Ln}9,$$
$$BmL^2 < 2B\delta_2 Ln \leq \frac{4\varep Ln}{81},$$
and
$$\frac{\delta_1 n^2}{2m} < \frac{\delta_1 Ln}{2\delta_2} \leq \frac{\varep Ln}{6}.$$
\end{proof}

The following proposition shows that all representations of spin groups which do not arise from representations of the
corresponding orthogonal groups are of such high degree that our character estimates are trivial.
Our general references for Deligne-Lusztig theory are \cite{C} and \cite{DM}. 

\begin{propo}\label{spin1}
Let $q$ be a power of an odd prime $p$ and let $\eps = \pm$. Then the following statements hold.

\begin{enumerate}[\rm(i)]
\item Suppose $\chi$ is a faithful irreducible character of $\Spin_{2n+1}(q)$ with $n \geq 2$. Then 
$$\chi(1) > q^{n(n+1)/2}/4.$$
\item Suppose $\chi$ is an irreducible character of $\Spin^\eps_{2n}(q)$ with $n \geq 3$, which is not obtained 
by inflating an irreducible character of $\Om^\eps_{2n}(q)$. Then 
$$\chi(1) > q^{n(n-1)/2}/4.$$
\end{enumerate}
\end{propo}

\begin{proof}
(i) For $G = \Spin_{2n+1}(q)$, the dual group $G^*$ is the projective conformal symplectic 
group $G^* = \PCSp_{2n}(q)$, which is the quotient of $\tilde G = \CSp_{2n}(q)$ by its 
center $\ZB(\tilde G)$. Suppose that $\chi$ belongs to the rational Lusztig series labeled by 
the $G^*$-conjugacy class of a semisimple element $s^* \in G^*$. By assumption, $\Ker(\chi) \cap \ZB(G) = 1$. 
Hence $s^* \notin [G^*,G^*]$ by \cite[Proposition 4.5]{NT}, which means that the conformal coefficient of
an inverse image $s \in \tilde G$ of $s^*$ is not a square in $\FF_q^\times$. Now the computations on p. 1188
of \cite{Ng} shows that
$$\chi(1) \geq [G^*:\CB_{G^*}(s^*)]_{p'} \geq [\tilde G:\CB_{\tilde G}(s)]_{p'}/2 > q^{n(n+1)/2}/4.$$

(ii) Here we have that $G = \Spin^\eps_{2n}(q) = \Spin(V)$, where 
$V = \FF_{q}^{2n}$ be endowed with a non-degenerate quadratic form $Q$. 
We recall some basic facts from spinor theory, cf.\ \cite{Ch}. The 
Clifford algebra $\CLF(V)$ is the quotient of the tensor 
algebra $T(V)$ by the ideal $I(V)$ generated by $v \otimes v - Q(v)$, $v \in V$. 
The natural grading on $T(V)$ passes over to $\CLF(V)$ and allows 
one to write $\CLF(V)$ as the direct sum of 
its even part $\CLF^{+}(V)$ and odd part
$\CLF^{-}(V)$. We denote the identity element of $\CLF(V)$ by $e$. 
The algebra $\CLF(V)$ admits a canonical 
anti-automorphism $\al$, which is defined via 
$$\al(v_{1}v_{2} \ldots v_{r}) = v_{r}v_{r-1} \ldots v_{1}$$ 
for $v_{i} \in V$. The Clifford group $\Gamma(V)$ is the group of all 
invertible $s \in \CLF(V)$ such that $sVs^{-1} \subseteq V$. 
The action of $s \in \Gamma(V)$ on $V$ defines a surjective homomorphism
$\phi~:~\Gamma(V) \to \GO(V)$ if $m$ is even, and $\phi~:~\Gamma(V) \to \SO(V)$ 
if $m$ is odd, with $\Ker(\phi) \geq \FF_{q}^{\times}e$. 
The special Clifford group
$\Gamma^+(V)$ is $\Gamma(V) \cap \CLF^{+}(V)$. 
Let 
$$\Gamma_0(V) := \{ s \in \Gamma(V) \mid \al(s)s = e\}.$$ 
The reduced Clifford group, or the spin group, is then defined to be 
$$G= \Spin(V) = \Gamma^+(V) \cap \Gamma_0(V),$$ 
and the following sequence 
$$1 \longrightarrow \langle -e \rangle \longrightarrow \Spin(V) 
    \stackrel{\phi}{\longrightarrow} \Omega(V) \longrightarrow 1$$
is exact.

Now fix $v \in V$ with $Q(v) \neq 0$ and write $V = A \oplus \langle v\rangle_{\FF_q}$ with 
$A = v^\perp$. Let $\CLF_A$ the subalgebra of $\CLF(V)$ generated by all $a \in A$, 
and let $\CLF(A)$ denote the Clifford algebra of the quadratic space $(A,Q|_A)$.
Then, by \cite[Lemma 4.1]{LBST3}, there is a canonical isomorphism 
$\CLF(A) \cong \CLF_A$ which induces a group isomorphism $\Spin(A) \cong \CLF_A \cap \Spin(V )$,
and the following sequence 
$$1 \longrightarrow \langle -e \rangle \longrightarrow \CLF_A \cap \Spin(V) 
    \stackrel{\phi}{\longrightarrow} \Omega(A) \longrightarrow 1$$
is exact.    

By assumption, $-e \notin \Ker(\chi)$. It follows that the restriction of $\chi$ to $\Spin(A) \cong \Spin_{2n-1}(q)$
contains an irreducible constituent $\rho$ with $-e \notin \Ker(\rho)$. By (i) we now have
$$\chi(1) \geq \rho(1) > q^{n(n-1)/2}/4.$$
\end{proof}


\begin{proof}[Proof of Theorem \ref{main1}]
We have already treated the classical groups $G$ in Theorem~\ref{classical-cases}.  
When $n$ is bounded, if $\delta$ is chosen small enough, there are no elements $g$
with $|\CB_G(g)| \le |G|^\delta$.
When $n$ is sufficiently large in terms of $\delta$,
the case of spin groups follows from Proposition~\ref{spin1} and Theorem \ref{classical-cases} for $\Omega^{\pm}_m(q)$.  
\end{proof}

\section{Effective bounds on character values}

We begin with the following statement that handles the {\it irreducible Weil characters} of finite symplectic groups, see eg. \cite{GMT}:

\begin{lemma}\label{weil}
Let $q$ be an odd prime power, $n \geq 9$, and let $G:=\Sp_{2n}(q)$ with natural module $V = \FF_q^{2n}$. For any element $g \in G$, let 
$$e(g) := \max( \dim \Ker(g-1_V), \dim \Ker(g+1_V)).$$
Then the following statements hold for $\chi$ any of the four irreducible Weil characters of $G$.
\begin{enumerate}[\rm(i)]
\item If $e(g) \leq 5$, then $|\chi(g)| < \chi(1)^{3/n}$.
\item If $e(g) \geq 6$ and $|\CB_G(g)| \leq q^{n^2\delta}$ for some $\delta > 0$, then 
$|\chi(g)| < \chi(1)^{9\sqrt{\delta}/8}$.
\end{enumerate} 
\end{lemma}

\begin{proof}
Note that $\chi(1) = (q^n \pm 1)/2 > q^{n-1} \geq q^{8n/9}$. 
Furthermore, for $k := e(g)$ we have by Theorem 2.1 and Lemma  3.1 of \cite{GMT}
that $|\chi(g)| \leq q^{k/2}$, and (i) follows immediately. Assume now that $k \geq 6$ and $|\CB_G(g)| \leq q^{n^2\delta}$. Then 
$$q^{n^2\delta} \geq |\CB_G(g)| \geq q^{(k^2-3k)/2} \geq q^{k^2/4}$$
by Proposition \ref{cent}, and so $k \leq 2n\sqrt{\delta}$, and (ii) follows.
\end{proof}

\begin{theor}\label{main-bound2}
For every $\gam$ with $4/5 < \gamma < 1$, there exists an explicit constant $\delta > 0$ such that the following statement holds.
Let $q$ be any prime power, $n \geq 9$, and let 
$$G \in \{ \Sp_{2n}(q),\SO^+_{2n}(q),\SO^-_{2n+2}(q),\Omega^-_{2n+2}(q)\} \cup \{\SO_{2n+1}(q) \mbox{ with }2\nmid q,\Om^+_{2n}(q) \mbox{ with }2|q\}.$$
Suppose that $g \in G$ satisfies $|\CB_G(g)| \leq q^{n^2\delta}$. Then
$$|\chi(g)| \leq 4\cdot\chi(1)^\gam$$
for all $\chi \in \Irr(G)$. In particular, if $\gamma = 0.99$ then one can take $\delta = 0.0011$.
\end{theor}

\begin{proof}
(i) For any fixed $\gamma \in (4/5,1)$, we can find 
\begin{equation}\label{for-delta0}
  0< \delta \leq \min\left(\frac{\gam}{4},\frac{1-\gam}{1.4}\right)
\end{equation}  
such that
\begin{equation}\label{for-delta}
  \sqrt{A\frac{\delta}{2\gam}} + \frac{\delta}{1-\gam} + 4(1-\gam) \leq \gam,
\end{equation}
where $A$ can be taken to be $1216$. (For instance, one can take 
$\delta = 0.00036$ if $\gam = 0.9$, and $\delta = 0.0011$ if $\gam = 0.99$.)
Let $\chi(1) = q^{nL}$ (so again $L \geq 0$ is again not necessarily an integer), and we aim to show that
$$|\chi(g)| \leq \chi(1)^\gam.$$ 
First we note that
$$|\chi(g)| \leq |\CB_G(g)|^{1/2} \leq q^{n^2\delta/2} \leq \chi(1)^\gam$$
if $L \geq n\delta/2\gam$.
So we may assume 
\begin{equation}\label{bound2-1}
  n \geq 9,~~L \leq n\delta/2\gam \leq n/8.
\end{equation}  
Note that this implies 
\begin{equation}\label{for-n1}
  n \geq \max(7.6L,7+1.4L).
\end{equation}
As the statement is obvious when $\chi(1)=1$, we may assume that $\chi(1) = q^{nL} > 1$, and so 
$L > 1/2$ by \cite[Theorem 1.1]{TZ1}. In fact, $L > 1$ unless $G = \Sp_{2n}(q)$ with $2 \nmid q$ and 
$\chi$ is a Weil character, in which case we can apply Lemma \ref{weil}.

\smallskip
(ii) We will choose some integer $m \geq 1$ later (see \eqref{for-m1}). Decompose
\begin{equation}\label{for-d1}
  (\chi\bar\chi)^m = \sum^t_{i=1}a_i\chi_i,
\end{equation}  
where $\chi_i \in \Irr(G)$ and $a_i \in \ZZ_{>0}$. We will bound $\sum_ia_i = \Csum((\chi\bar\chi)^m,G)$ by 
restricting $\chi$ to $L_n:=\GL_n(q)$:
$$\chi|_{L_n} = \sum_{\al \in \Irr(L_n),~\al(1) \leq q^{nL}}b_\al\al,$$ 
so that 
\begin{equation}\label{for-d2}
  (\chi\bar\chi)^m|_{L_n} = \sum_{\al_i \in \Irr(L_n),~1 \leq i \leq 2m}b_{\al_1} \ldots b_{\al_{2m}}\al_1 \ldots \al_m\bar\al_{m+1}
  \ldots \bar\al_{2m}.
\end{equation}
The choice \eqref{for-m1} of $m$ implies by \eqref{for-delta0} that
$$5.6mL \leq \frac{5.6n\delta}{4(1-\gam)} \leq n.$$ 
With $\al_i \in \Irr(L_n)$ of degree $\leq q^{nL}$ and 
$$\beta := \al_1 \ldots \al_m\bar\al_{m+1}\ldots \bar\al_{2m},$$ 
we then have by Proposition \ref{tensor}(ii) that
$$\Csum(\beta,L_n) \leq [\beta,\beta]_{L_n} \leq 
    8q^{8m^2L^2},$$
It now follows from \eqref{for-n1}, \eqref{for-d2}, and Theorem \ref{restr-gl} and Corollary \ref{restr-gl2}  that
$$\Csum((\chi\bar\chi)^m,L_n) \leq 8q^{8m^2L^2}(\sum_\al b_\al)^{2m} 
  = 8q^{8m^2L^2}\Csum(\chi,L_n)^{2m} \leq 8q^{2m\sqrt{AnL^3}+8m^2L^2}.$$
Thus 
$$\sum_ia_i = \Csum((\chi\bar\chi)^m,G) \leq \Csum((\chi\bar\chi)^m,L_n) \leq 8q^{2m\sqrt{AnL^3}+8m^2L^2}.$$

\smallskip
(iii) Now we rewrite \eqref{for-d1} as 
\begin{equation}\label{bound2-chi}
  (\chi\bar\chi)^m = \sum^s_{i=1}a_i\al_i + \sum^t_{j=1}b_j\beta_j,
\end{equation}  
where $a_i, b_j \in \ZZ_{>0}$, $\al_i,\beta_j \in \Irr(G)$, and 
$$\al_i(1) < q^{n^2\delta},~~\beta_j(1) \geq q^{n^2\delta}.$$
Then the result of (ii) can be written as 
$$\sum^s_{i=1}a_i + \sum^t_{j=1}b_j \leq 8q^{2m\sqrt{AnL^3}+8m^2L^2}.$$
Using the bounds $|\al_i(g)| \leq \al_i(1)$, $8^{1/2m} < 3$, and 
$$|\beta_j(g)| \leq |\CB_G(g)|^{1/2} \leq q^{n^2\delta/2} \leq \beta_j(1)/q^{n^2\delta/2},$$ 
we have
\begin{equation}\label{bound2-4}
  \begin{split}
  \biggl|\sum^s_{i=1}a_i\al_i(g)\biggl|^{1/2m}  & \leq \biggl(\sum^s_{i=1}a_i\al_i(1)\biggl)^{1/2m} \leq 3q^{\sqrt{AnL^3}+4mL^2+n^2\delta/2m},\\
  \biggl|\sum^t_{j=1}b_j\beta_j(g)\biggl|^{1/2m} & \leq \left(\frac{\sum^t_{j=1}b_j\beta_j(1)}{q^{n^2\delta/2}}\right)^{1/2m} \leq 
  \frac{\chi(1)}{q^{n^2\delta/4m}}.
 \end{split} 
\end{equation}  
Recall that we assume $\gam > 4/5$. It follows that 
$$L \leq \frac{n\delta}{2\gam} < \frac{n\delta}{8(1-\gam)}$$
and so there exists an integer $m \geq 1$ such that 
\begin{equation}\label{for-m1}
  \frac{n\delta}{8(1-\gam)} \leq mL \leq \frac{n\delta}{4(1-\gam)}.
\end{equation}
Our choice \eqref{for-m1} of $m$ implies that $n\delta/mL \geq 4(1-\gam)$ and so
\begin{equation}\label{bound2-7}
  |\sum^t_{j=1}b_j\beta_j(g)|^{1/2m} \leq \frac{\chi(1)}{q^{n^2\delta/4m}} \leq \chi(1)^\gam.
\end{equation}  
Next, $n^2\delta/2m \leq 4nL(1-\gam)$ and 
$mL^2 \leq nL\delta/4(1-\gam)$. It follows from \eqref{for-delta} that 
$$\sqrt{AnL^3}+4mL^2+\frac{n^2\delta}{2m} \leq nL\left( \sqrt{A\frac{\delta}{2\gam}} + \frac{\delta}{1-\gam} + 4(1-\gam)\right) \leq nL\gam.$$ 
Hence \eqref{bound2-4} implies that
\begin{equation}\label{bound2-6}
  |\sum^s_{i=1}a_i\al_i(g)|^{1/2m} \leq 3\chi(1)^\gam.
\end{equation}  
Combining \eqref{bound2-chi}, \eqref{bound2-6} and \eqref{bound2-7} together, and recalling $m \geq 1$, we arrive at
$$\begin{aligned}|\chi(g)| & \leq \left(|\sum^s_{i=1}a_i\al_i(g)|+|\sum^t_{j=1}b_j\beta_j(g)|\right)^{1/2m} \\
    & \leq |\sum^s_{i=1}a_i\al_i(g)|^{1/2m} +  |\sum^t_{j=1}b_j\beta_j(g)|^{1/2m} \leq 4\chi(1)^\gam,\end{aligned}$$
as stated.
\end{proof}

Note that the following statement in fact also applies to irreducible $\ell$-Brauer characters for any $\ell$ coprime to 
$2q$, with the proof in the modular case following the proof of \cite[Theorem 3.9]{LBST2}. We restrict ourselves
to the complex case (and also note that \cite[Theorem 3.9]{LBST2} needs the assumption that either $\ell=0$ or 
$\ell \nmid (q \cdot \gcd(n,q+1))$). 

\begin{propo}\label{spin2}
Let $q$ be a power of an odd prime $p$ and let $\eps = \pm$. Then the following statements hold.

\begin{enumerate}[\rm(i)]
\item Suppose $\chi$ is an irreducible character of $G=\SO_{2n+1}(q)$ with $n \geq 2$ and 
$\chi$ is reducible over $[G,G]=\Om_{2n+1}(q)$. Then 
$$\chi(1) > q^{n^2/2}.$$
\item Suppose $\chi$ is an irreducible character of $G=\SO^\eps_{2n}(q)$ with $n \geq 4$ and $\chi$ is reducible 
over $[G,G]=\Om^\eps_{2n}(q)$. Then 
$$\chi(1) > (q-1)q^{n(n-1)/2-1}.$$
\end{enumerate}

\end{propo}

\begin{proof}
(a) Assume that $\chi$ belongs to the rational Lusztig series $\EC(G,(s))$ labeled by a semisimple element 
$s \in G^*$, where the dual group $G^*$ is $\Sp_{2n}(q)$, respectively $\SO^\eps_{2n}(q)$. By assumption,
$\chi$ is reducible over $[G,G]$, which has index $2$ in $G$. It follows by Clifford's theorem that 
$\chi = \chi\lambda$, where $\lambda$ is the unique non-principal irreducible character of 
$G/[G,G]$. Note that $\lambda$ is the semisimple character corresponding to the 
central involution $z \in G^*$. Furthermore, according to \cite[Proposition 13.30]{DM} and 
its proof, the tensor product with $\lambda$ defines a bijection between
the series $\EC(G,(s))$ and $\EC(G,(sz))$. Since 
$\chi = \chi \lambda$, we conclude that $s$ and $sz$ 
are conjugate in $G$. As $\chi(1)$ is divisible by $N:= [G^*:\CB_{G^*}(s)]_{p'}$, it suffices to show 
that $N$ satisfies the lower bounds mentioned in the statements.

\smallskip
(b) Consider the case $G^* = \Sp_{2n}(q)$. It is easy to see that the condition $s$ and $sz$ are $G^*$-conjugate
implies that $1$ and $-1$ should occur as eigenvalues of $s$ with the same even multiplicity say $2a \geq 0$, and 
$$\CB_{G^*}(s) \cong \Sp_{2a}(q) \times \Sp_{2a}(q) \times \prod^c_{i=1}\GL_{k_i}(q^{a_i}) \times \prod^d_{j=1}\GU_{l_j}(q^{b_j}),$$
where $a_i,b_j \geq 0$ and $\sum_i k_ia_i +\sum_j l_jb_j = n-2a$. It follows that 
$$N \geq [\Sp_{2n}(q):(\Sp_{2a}(q) \times \Sp_{2a}(q) \times \GU_{n-2a}(q))]_{p'}.$$
If $a=n/2$, then 
\begin{equation}\label{cent-1}
  N \geq [\Sp_{2n}(q):\Sp_n(q)^2]_{p'} = \prod^{n/2}_{i=1}\frac{q^{n+2i}-1}{q^{2i}-1} > q^{n^2/2}.
\end{equation}  
If $0 \leq a < n/2$, then by \cite[Lemma 6.1(iii)]{GLT} we have 
\begin{equation}\label{cent-2}
  N \geq \prod^{n-2a}_{i=1}(q^i+(-1)^i) \cdot \frac{(q^{2n-4a+2}-1) \ldots (q^{2n}-1)}
      {((q^2-1) \ldots (q^{2a}-1))^2}  > (q-1)q^{M-1},
\end{equation}      
where 
$$M = \frac{(n-2a)(n-2a+1)}{2} + a(4n-6a) = \frac{n^2+(4a+1)(n-2a)}{2} \geq (n^2+1)/2.$$
Since $q-1> q^{1/2}$ when $q \geq 3$, we obtain that $N \geq q^{n^2/2}$ in this case as well.

\smallskip
(c) Suppose that $G^* = \SO^\eps_{2n}(q)$. The condition $s \in \SO^\al_{2n}(q)$ and $sz$ are $G^*$-conjugate
imply that $1$ and $-1$ should occur as eigenvalues of $s$ with the same even multiplicity say $2a \geq 0$.
First suppose that $a=0$. Then 
$$\CB_{G^*}(s) \cong \prod^c_{i=1}\GL_{k_i}(q^{a_i}) \times \prod^d_{j=1}\GU_{l_j}(q^{b_j}),$$
where $a_i,b_j \geq 0$ and $\sum_i k_ia_i +\sum_j l_jb_j = n$. It follows that 
$$N \geq [\SO^+_{2n}(q):\GU_n(q)]_{p'} \geq \frac{(q-1)(q^2+1) \ldots (q^n+(-1)^n)}{q^n+1}.$$
Using \cite[Lemma 6.1(iii)]{GLT} and observing that $(q^2+1)(q^3-1)>q(q^4+1)$, it is easy to
check that 
$$N > (q-1)q^{n(n-1)/2-1}$$
in this case. Suppose that $a=n/2$. Then
$$\CB_{G^*}(s) \cong (\SO^\al_n(q) \times \SO^\al_n(q)) \cdot C_2$$
for some $\al = \pm$. Using \eqref{cent-1} we then get  
$$N \geq [\SO^+_{2n}(q):\SO^-_n(q)^2]_{p'}/2 = \frac{(q^{n/2}-1)^2}{2(q^n+1)} \cdot [\Sp_{2n}(q):\Sp_n(q)^2]_{p'} > q^{n^2/2-1}.$$
Assume now that $1 \leq a < n/2$. Then 
$$\CB_{G^*}(s) \cong \prod^c_{i=1}\GL_{k_i}(q^{a_i}) \times \prod^d_{j=1}\GU_{l_j}(q^{b_j}) 
    \times (\SO^\al_{2a}(q) \times \SO^\al_{2a}(q)) \cdot C_2,$$
where $\al = \pm$, $a_i,b_j \geq 0$ and $\sum_i k_ia_i +\sum_j l_jb_j = n-2a$. It follows that    
$$N \geq \frac{|\SO^+_{2n}(q)|_{p'}}{2|\SO^-_{2a}(q)^2 \times \GU_{n-2a}(q)|_{p'}} = 
    \frac{(q^a-1)^2}{2(q^n+1)} \cdot [\Sp_{2n}(q):(\Sp_{2a}(q)^2 \times \GU_{n-2a}(q))]_{p'}.$$
Using \eqref{cent-2} we obtain that  
$$N \geq \frac{(1-1/q)^3}{2(1+1/q^4)}q^{M-(n-2a)}$$
with
$$M-(n-2a) = \frac{n^2+(4a-1)(n-2a)}{2} \geq (n^2+3)/2.$$
As $q \geq 3$, it follows that $N > (0.76)q^{n^2/2}$ in this case.
\end{proof}

\begin{theor}\label{main2a}
For every $\varep$ with $4/5 < \varep < 1$, there exists an explicit constant $\delta > 0$ such that the following statement holds.
Let $q$ be any prime power, $n \geq 9$, and let $G \in \CLS_n(q)$.
Suppose that $g \in G$ satisfies $|\CB_G(g)| \leq q^{n^2\delta}$. Then
$$|\chi(g)| \leq 4 \cdot \chi(1)^\varep$$
for all $\chi \in \Irr(G)$. For instance, if $\varep = 0.99$, one can take $\delta = 0.0011$.
\end{theor}

\begin{proof}
First we show that the statement holds for all $G$ with $\SL^\eps_n(q) \lhd G \leq \GL^\eps_n(q)$. Let $\chi \in \Irr(G)$ lie under
an irreducible character $\theta$ of $\GL^\eps_n(q)$. If $\chi  = \theta|_G$, then we are done by applying 
\cite[Theorem 1.4]{GLT} to $\theta$. Otherwise $\theta|_G$ is reducible, whence $\theta$ is reducible over $\SL^\eps_n(q)$ 
and so $\chi(1)$ is very large by \cite[Lemma 8.3]{GLT}, in which case the statement holds  
by the centralizer bound $|\chi(g)| \leq |\CB_G(g)|^{1/2}$.

For all the remaining groups in $\CLS_n(q)$ but spin groups, we are done by Theorem \ref{main-bound2}.
It remains to consider the case $2 \nmid q$ and $G = \Spin^\eps_m(q)$ with 
$m = 2n+1$, or $(m,\eps) = (2n,+)$, $(2n+2,-)$. By Proposition \ref{spin1}, either $\chi$ is obtained by 
inflating an irreducible character of $\Om^\eps_m(q)$, or $\chi(1)$ is very large, in which case Theorem \ref{main2a} obviously holds, again by the centralizer bound $|\chi(g)| \leq |\CB_G(g)|^{1/2}$.
So we may assume that $\chi \in \Irr(\Om^\eps_m(q))$. Now, by Proposition \ref{spin2}, either 
$\chi$ extends to $\SO^\eps_m(q)$, or $\chi(1)$ is very large. In the former case, we are done by applying
Theorem \ref{main-bound2} to $\SO^\eps_m(q)$. In the latter case, we are again done by
the centralizer bound.
\end{proof}


\begin{proof}[Proof of Theorem \ref{main2}]
For any given $\varep$, we can choose some $\varep^*$ so that $4/5 < \varep^* < \varep$ (say, $\varep^* = \varep/2+2/5$). In fact,
for $\varep = 0.992$, we will choose $\varep^* = 0.99$. Now we apply Theorem \ref{main2a} to get an explicit $\delta^* > 0$ (which can be taken 
to be $0.0011$ if $\varep^* = 0.99$), such that 
\begin{equation}\label{e-star1}
  |\chi(g)| \leq 4 \cdot \chi(1)^{\varep^*}
\end{equation}   
for all $G \in \CLS_n(q)$, for all $\chi \in \Irr(G)$, and for all $g \in G$ with $|\CB_G(g)| \leq q^{n^2\delta^*}$. 

Next we choose 
\begin{equation}\label{e-star2}
  \delta:=\min\bigl(\delta^*,\frac{16}{25}\varep(\varep-\varep^*)\bigr),
\end{equation}  
which is $0.0011$ when $(\varep,\varep^*,\delta^*) = (0.992,0.99,0.0011)$.
Consider any $g \in G$ with $|\CB_G(g)| \leq q^{n^2\delta}$ and any $\chi \in \Irr(G)$. If $\chi$ is linear
then $|\chi(g)| = \chi(1)^\varep$. Assume $\chi(1) > 1$. An application of \cite[Theorem 1.1]{TZ1} shows that
\begin{equation}\label{e-star3}
  \chi(1) > q^{4n/5}.
\end{equation}  
Now if $n \geq 2.5/(\varep-\varep^*)$, then
$$\chi(1)^{\varep-\varep^*} > 2^{(4n/5)(\varep-\varep^*)} > 4,$$
and so \eqref{e-star1} implies that $|\chi(g)| \leq \chi(1)^\varep$. Finally, if $n < 2.5/(\varep-\varep^*)$, then the choice \eqref{e-star2} implies
that
$$n\delta < \frac{5}{2(\varep-\varep^*)} \cdot  \frac{16}{25}\varep(\varep-\varep^*) = \frac{8\varep}{5}.$$
Combining with \eqref{e-star3}, this yields
$$|\chi(g)| \leq |\CB_G(g)|^{1/2} \leq q^{n^2\delta/2} < q^{4n\varep/5} < \chi(1)^{\varep},$$
completing the proof.
\end{proof}

In the following application, for a finite group $S$ and a fixed element $g \in S$ we consider the conjugacy class $C = g^S$ and random walks on
the (oriented) Cayley graph $\Gamma(S,C)$ (whose vertices are $x \in S$ and edges are $(x,xh)$ with $x \in S$ and 
$h \in C$). Let $P^t(x)$ denote the probability that a random product of $t$ conjugates of $g$ is equal to $x \in S$, 
and let $U(x) := 1/|S|$ denote the uniform probability distribution on $S$. Also, let
$$||P^t-U||_\infty := |S| \cdot \max_{x \in S}|P^t(x)-U(x)|.$$

\begin{corol}\label{walk}
There exists an absolute constant $1 > \gam > 0$ such that the following statements
hold. Let $S \in \CLS_n(q)$ be a quasisimple group with $n \geq 9$, and let $g \in S$ be such that $|\CB_S(g)| \leq |S|^\gam$, and let 
$C = g^S$. 
\begin{enumerate}[\rm(i)]
\item If $t \geq 3$, then $P^t$ converges to $U$ in the $||\cdot||_{\infty}$-norm when $q\to \infty$; in particular, the 
Cayley graph $\Gamma(S,C)$ has diameter at most $3$ when $q$ is sufficiently large.
\item The mixing time $T(S,C)$ of the random walk on $\Gamma(S,C)$ is  at most $2$ for $q$ sufficiently large.
\end{enumerate}
\end{corol}

\begin{proof}
We follow the proof of \cite[Theorem 1.11]{BLST}.
Consider the {\it Witten $\zeta$-function}
\begin{equation}\label{witten}
  \zeta^S(s) = \sum_{\chi \in \Irr(S)}\frac{1}{\chi(1)^s}.
\end{equation}  
By \cite[Theorem 1.1]{LS}, $\lim_{q \to \infty}\zeta^S(s) = 1$ as long as $s > 2/h$, where $h$ is the Coxeter number, i.e. 
$h := n$ if $S = \SL^\eps_n(q)$,
$h:=2n$ if $S = \Sp_{2n}(q)$, $\Spin_{2n+1}(q)$ or $\Spin^-_{2n+2}(q)$, and $h := 2n-2$ if $S = \Spin^+_{2n}(q)$. 

For (i), we apply Theorem \ref{main1} with $\varep = 1/4$. Choosing $\gamma$ suitably, we can ensure that
$|\CB_S(g)| \leq q^{n^2\delta}$, and so $|\chi(g)| \leq \chi(1)^{1/4}$ for all $\chi \in \Irr(S)$. Now
we have by a well-known result (see \cite[Chapter 1, 10.1]{AH}) that
$$||P^t-U||_\infty \leq \sum_{1_S \neq \chi \in \Irr(S)} \left( \frac{|\chi(g)|}{\chi(1)}\right)^t\chi(1)^2 \leq \zeta^S(3t/4-2)-1.$$
Now, as $n \geq 9$, if $t \geq 3$ then $3t/4-2 > 2/h$, and so the statement follows.

For (ii), note that $P^t(x)$ is the probability that a random walk on the Cayley graph $\Gamma(S,C)$ reaches $x$ after $t$ steps. Let
$$||P^t-U||_1 := \sum_{x \in S}|P^t(x)-U(x)|.$$ 
For (i), we apply Theorem \ref{main1} with $\varep = 1/3$. Choosing $\gamma$ suitably, we can ensure that
$|\CB_G(g)| \leq q^{n^2\delta}$, and so $|\chi(g)| \leq \chi(1)^{1/3}$ for all $\chi \in \Irr(S)$.
By the Diaconis-Shahshahani bound \cite{DS},
$$(||P^t-U||_1)^2 \leq \sum_{1_S \neq \chi \in \Irr(S)} \left( \frac{|\chi(g)|}{\chi(1)}\right)^{2t}\chi(1)^2 \leq \zeta^S(4t/3-2)-1.$$
As $n \geq 9$, if $t \geq 2$ then $4t/3-2 > 2/h$, and the statement follows.
\end{proof}

Lubotzky conjectured in \cite[p. 179]{Lub} that, if $S$ is a finite simple group and $C$ is a non-trivial conjugacy class
of $S$, then the mixing time of the random walk on $\Gamma(G, C)$ is
linearly bounded above in terms of the diameter of $\Gamma(G, C)$. Corollary \ref{walk} confirms this conjecture for 
the medium-size conjugacy classes $C$.

\section{Product one varieties}

Let $\uG$ be a simply connected simple algebraic group over an algebraically closed field $\FF$
of characteristic $p \ge 0$, of rank $r$. 
Let $m \ge 3$ be a positive integer and
let $\uC_1, \ldots, \uC_m$ be nontrivial conjugacy classes in $\uG$.  Let
$$\uX = \uX(\uC_1, \ldots, \uC_m)  := \{ (g_1, \ldots,g_m) \in \uC_1\times \cdots\times\uC_m \mid \prod^m_{i=1}g_i = 1\}$$
denote the closed subvariety of $\uC_1\times \cdots\times\uC_m$
consisting of $m$-tuples with product $1$ in $\uG$.

We will prove: 

\begin{theorem}  
\label{product-one}   
There exists some absolute constant $C$ such that if  either $m \ge 7$
or $r \geq C$, the following statements hold for any simple 
simply connected algebraic group $\uG$ of rank $r$ and any $m \geq 3$ noncentral
 conjugacy classes $\uC_1, \ldots,\uC_m$ 
in $\uG$:
\begin{enumerate}[\rm(i)]
\item $\dim \uX(\uC_1, \ldots, \uC_m) \le (m-1)\dim \uG - mr$.
\item Equality holds in {\rm (i)} if and only if all the $\uC_i$ are regular classes.
\item If the equivalent conditions in {\rm (ii)} hold, then $\uX(\uC_1, \ldots, \uC_m)$ is irreducible. 
\end{enumerate}
\end{theorem}

We first see that it suffices work over the algebraic closure of a finite field.   In fact,
we conjecture that the result holds for any $m \ge 3$ and $r \ge 2$. 

\begin{lemma}
\label{conj-class}
Let $\cG$ be a group scheme of finite type over an irreducible base $\cS$ of finite type over $\ZZ$.  Let $\cX$ denote a closed subscheme of $\cG$ which maps isomorphically to $\cS$.  Then there exists a dense open subset $U$ of $\cS$ and a (locally closed) subscheme $\cC$ of $\cG$
such that for each $u\in U$, the fiber $\cC_u$ is the conjugacy class of $\cX_u$ in $\cG_u$.
\end{lemma}

\begin{proof}
Let $\pi\colon \cW\to \cS$ denote a morphism of finite type, and $C$ a  constructible subset of $\cW$.  
Then the closure of the generic fiber $C_\eta := C\cap \pi^{-1}(\eta)$ in $\cW_\eta$ coincides with $(\bar C)_\eta$.
Indeed, we can reduce to the case that $C$ is dense in $\cW$, in which case it contains all generic points of $\cW$ and in particular all generic points of $\cW_\eta$,
so $C_\eta$ is dense in $\cW_\eta$.

By \cite[Corollaire 9.5.4]{EGA}, if the generic fiber of the constructible set $C$ is locally closed, then 
there exists a dense open subset $U_1$ of $\cS$ such that for all $u\in U_1$, the fiber $C_u$ is locally closed.  
Moreover, if $C'$ is a second constructible subset of $\cW$ such that $C_\eta = C'_\eta$, then there exists a dense open subset $U_2$ of $\cS$ such that for all $u\in U_2$,
$C_u = C'_u$ \cite[Corollaire 9.5.2]{EGA}.  
Given $C$, we may define a locally closed set $C'$ as follows: $C' := \bar C \setminus \bar D$, where $D:=\bar C\setminus C$.
Since for constructible sets, Zariski closure commutes with passage to the generic fiber, $C'_\eta = C_\eta$, and so $C_u$ is locally closed for $u\in U := U_1\cap U_2$.

We apply this to the image $C$ of the conjugation map $\pi \colon \cW := \cG\times_{\cS} \cX \to \cG$, which is constructible by Chevalley's theorem. 
The generic fiber $C_\eta$ is the conjugacy class of the point $\cX_\eta$ in the algebraic group $\cG_\eta$ and is therefore locally closed.
Let $\cU$ be the open subscheme of $\cS$ with underlying set $U$; replacing $\cS$ by $\cU$, we may assume without loss of generality that $C=C'$ is locally closed.
Let $\bar\cC$ denote the reduced closed subscheme of $\cG$ whose underlying set is $\bar C$ in $\cG_U$, and let $\cC$ denote the open subscheme of $\bar\cC$ whose underlying set is $\cC\setminus D$. Then $\cC_u = C_u$ for all $u\in \cU$, so the lemma holds.
\end{proof}

\begin{proposition}  \label{finite fields}
Theorem~\ref{product-one} holds over all algebraically closed fields $\FF$ if it holds over $\overline{\FF}_p$ for all $p$.
\end{proposition}

\begin{proof}
There exists a subfield $\FF_0$ of $\FF$, finitely generated over the prime field of $\FF$, such that $\uG$ and all the $\uC_i$ are defined over $\FF_0$.
Passing to a finite extension if necessary, we assume that $\uG$ is split over $\FF_0$ and each $\uC_i$ has a point $x_i$ defined over $\FF_0$.
Thus, there exists a finitely generated $\ZZ$-algebra $A$,
a split simple, simply connected group scheme $\cG$ over $A$, and $A$-valued points $\calx_i$ in $\cC_i$, such that the field of fractions of $A$ is $\FF_0$,
and extending scalars from $A$ to $\FF_0$ takes $\cG$ (resp. $\calx_i$) to $\uG$ (resp. $x_i$).  By Lemma~\ref{conj-class}, replacing $\Spec A$
with a dense open subscheme $\cU$, we may define locally closed subschemes $\cC_i$ of $\cG$ such that the fibers of $\uC_i$ over any $u\in \cU$ are conjugacy classes
in a split, simply connected, semisimple algebraic group with the same root system as $\uG$.  
Let $\uG_u$ (resp. $\uC_{i,u}$) denote the fiber of $u$ of $\cG$ (resp. $\cC_i$), and let
$\uX_u$ denote the product one subvariety of $\uC_{1,u}\times \cdots \times \uC_{m,u}$.
By \cite[Corollaire 9.5.6]{EGA} and \cite[Proposition 9.7.8]{EGA}, replacing $\cU$ with a suitable open subscheme, we may assume that
the dimension of $\uX_u$ and the number of  irreducible geometric components of 
$\uX_u$ are the same as for $\uX$.

Replacing $\cU$ with a smaller open subscheme if necessary, we may assume it is affine.  Its coordinate ring is finitely generated over $\ZZ$ so every closed point has finite residue field.  Thus, we may replace $\FF$ with an algebraic closure of a finite field. 
\end{proof}

Henceforth, we assume $\FF = {\overline\FF}_p$.

We first point out the result for $\SL_2(\FF)$ with $\FF$ algebraically closed.
See \cite{GM, Sh}.   The result is a bit different, as is the proof. 

\begin{lemma} 
\label{rank one}
Let $\uG=\SL_2(\FF)$ and let $\uC_i, \ldots, \uC_m$ be noncentral conjugacy classes
of $\uG$ with $m \geq 3$.   Then  $\dim \uX = 2m -3 $ and
either $\uX$ is irreducible or $m=3$ and there exist $x_i \in \uC_i$ 
with product $1$ fixing a unique line (and the $\uC_i$ are not all unipotent).
\end{lemma}

\begin{proof}   By Proposition \ref{finite fields}, we may assume
that $\FF$ is the algebraic closure of a finite field.   We can work in $\GL_2(\FF)$
(this does not change the variety).  The advantage here is that there is a connected
center and it follows that there is an absolute bound on $|\chi(g)|$ for $g$ any
noncentral element and $\chi$ any irreducible character. 
  If $m \ge 4$, this follows  from inspection of the character tables
of $\GL_2(q)$ and (a variation of) Lemma \ref{sum} below.  

Assume that $m =3$.   A straightforward computation with $2 \times 2$ matrices
shows that for any $(x_1, x_2, x_3) \in \uX$  either the $x_i$ generate an irreducible
subgroup or there exist eigenvalues $a_i$ of $x_i$ with $a_1a_2a_3=1$.   Moreover,
if the latter case holds, then any triple in $\uX$ fixes a line. 

In the first case, as noted,  any triple
generates an irreducible subgroup of $\SL_2(\FF)$.  Then by Katz's rigidity
theorem (which holds in arbitrary characteristic \cite{Ka, SV}), all triples
are conjugate via an element of $\SL_2$, showing the variety is irreducible
of dimension $3$.

Otherwise any  $(x_1, x_2, x_3) \in \uX$ fixes a line.  If $\uB$ is a Borel subgroup, then
$\uX \cap \uB^3$ has two components (determined by the action of the $x_i$ on
the fixed line) or is irreducible if the  $C_i$ are  all unipotent.  These components are two
dimensional.     Conjugation then gives the result. 
\end{proof}

\begin{lemma}   
\label{dim-ineq}
For each integer $i\in [1,m]$, $\dim \uX \le \sum_{j \ne i} \dim \uC_j$.
\end{lemma}

\begin{proof}  The projection map from $\uX$ to $\prod_{j \ne i} \uC_j$
is injective.
\end{proof}


\begin{lemma}
\label{sum}
Suppose that 
$\dim \uX(\uC_1, \ldots, \uC_m) \ge (m-1)\dim \uG - mr$.
Then 
$$
\sum_i \dim \uC_i  \ge   m \dim \uG - \frac{m^2 r}{m-1}.
$$
Moreover if $x_i \in \uC_i$, then
$$
\sum_i \dim \CB_{\uG}(x_i)  \le \frac{m^2 r}{m-1}.
$$
\end{lemma}

\begin{proof}   Take the inequality above and sum over $i$.  This gives
$$
(m-1) \sum_i \dim \uC_i \ge m \dim \uX \ge m[(m-1)\dim \uG - mr].
$$

Thus, 
$$
\sum_i \dim \uC_i \ge m \dim \uG - \frac{m^2 r}{m-1},
$$
as claimed.  The second inequality follows. 
\end{proof}


\begin{lemma}   
\label{lem:small-error}  Suppose that there is some small $\eps > 0$ such that for all $q$ sufficiently large we have
\begin{equation}
\label{small-error}
\Bigm|\sum_{1_G \ne \chi \in \Irr(G)} \frac{\chi(x_1) \ldots \chi(x_m)}{\chi(1)^{m-2}} \Bigm|  < \epsilon,
\end{equation}
for $G = \uG(\FF_q)$. Then
$\uX$ is an irreducible variety of dimension 
$$e:=\sum (\dim \uC_i - \dim \uG).$$
\end{lemma}

\begin{proof}  By the Lang-Weil estimate, $|G| = (1+o(1))q^{\dim \uG}$ and $|C_i| = (1+o(1))q^{\dim \uC_i}$.
By the usual class product formula for $C_1,\ldots,C_m\subset G$, (\ref{small-error}) shows that the number
of $\FF_q$-points of $\uX^0$ is  
$(1+o(1))q^e$ whence by the Lang-Weil estimate, $\uX$ consists of a single irreducible component of dimension $e$, with
all other irreducible components of lower dimension.
On the other hand, it follows by intersection theory that any component
of $\uX$ has dimension at least $e$ (see for example the argument
in \cite{GM} or \cite{Sh}), whence the result. 
\end{proof}

We now prove that for $m \ge 7$,  Theorem \ref{product-one} holds for any $r$.

\begin{lemma}  
\label{unbounded}  Let $e$ be such that the minimal degree nontrivial
character of $\uG(q)$ has dimension at least a constant times $q^e$
(note that $e \ge r$).     Let $m_0$ be minimal such that
$$
\frac{m_0^2 r}{2m_0-2}  < e(m_0-2 - 2/h), 
$$
where $h$ is the Coxeter number of $\uG$.    Then
if $m \ge m_0$,   Theorem \ref{product-one} holds.  In particular, this holds
for $m \ge 7$.   
\end{lemma}

\begin{proof}   Let $x_i \in \uC_i(q)$ for $q$ large.  Using the trivial bound $|\chi(x_i)|  < 
|C(x_i)|^{1/2}$ where $C(x_i)$ is the centralizer of $x_i$ in the group of $\FF_q$-points
of $\uG$,  by Lemma \ref{sum} there is a constant $c$ such that

$$\frac{|\chi(x_1)\ldots\chi(x_m)|}{\chi(1)^{m-2}} \le c^mq^{\frac{m^2 r}{2m-2}}\chi(1)^{2-m}.$$

Thus, by our hypothesis and \cite[Theorem 1.1]{LS} for $q$ sufficiently large, 
$$
\Bigm|\sum_{1_G \ne \chi \in \Irr(G)} \frac{\chi(x_1) \ldots \chi(x_m)}{\chi(1)^{m-2}} \Bigm| 
<    \epsilon.
$$
It follows by Lemma \ref{lem:small-error}  that $\uX$ is irreducible of dimension  
$$\sum (\dim \uC_i - \dim \uG) \le  (m-1)\dim \uG -mr,$$ 
with equality if and only if the $\uC_i$ are all regular, whence the result holds.
\end{proof}

For the exceptional groups, the result works
for  smaller $m$.  Indeed if $\uG=E_8$, then this already gives the result for
any $m \ge 3$ since the smallest degree character has degree approximately $q^{29}$.
In certain classical groups, the result for any $m \ge 4$ holds as well by the previous
argument.

We can now prove that Theorem~\ref{product-one} holds for any $m \ge 3$
for sufficiently large rank.   Here is where we use the preceding results of this paper. 

\begin{proof}
We henceforth assume  that $\cG$ has rank $\ge 9$.
It is therefore of classical type.

Suppose that $\dim \uX \ge (m-1) \dim \uG - mr$.   Then by Lemma~\ref{dim-ineq},  
$$\sum_{j \ne i}  \dim \uC_j \ge (m-1) \dim \uG - mr.$$
This implies that $\dim \uC_i \ge \dim \uG - 2r$.  By the Lang-Weil estimate, $|C_i| \ge q^{\dim \uG - 2r}/2$ if $q$ is sufficiently large and so
 $|\CB_G(x_i)| \le 2q^{2r}$ for each $i$.  If $\uG$ is of sufficiently high rank, then
\cite[Corollary 8.5]{GLT} (in the case $G$ is of type A) and Theorem~\ref{main1} (for types B, C, and D),
imply that $|\chi(x_i)| \le \chi(1)^{1/4}$ for all irreducible characters $\chi$ of $G$.  This implies the sum in (\ref{small-error}) is bounded above by
$$\sum_{1_G\neq \chi\in \Irr(G)} \chi(1)^{-1/4},$$
which goes to zero as $|G|\to \infty$ among groups of Lie type of sufficiently high rank \cite[Theorem 1.1]{LS}.
This proves the theorem for simple algebraic groups over $\overline{\FF}_p$ and therefore for simple algebraic groups over any algebraically closed field.
\end{proof}

\end{document}